\documentclass{article}
\usepackage{amsmath,amssymb,amsthm, amsfonts, mathrsfs}
\usepackage{enumerate}
\usepackage{calligra,indentfirst,epsfig}
\calligra
\newtheorem{thm}{Theorem}[section]
\newtheorem{lem}{Lemma}[section]
\newtheorem{cor}{Corollary}[section]
\newtheorem{prop}{Proposition}[section]

\theoremstyle{remark}
\newtheorem{rem}{Remark}[section]

\theoremstyle{definition}

\allowdisplaybreaks
\begin{document}
\title{\bf {Repeated-root constacyclic codes over the \\ chain ring $\boldmath { \mathbb{F}_{p^m}[u]/\langle u^3 \rangle }$}}
\author{Anuradha Sharma{\footnote{Corresponding Author, Email address: anuradha@iiitd.ac.in}~ and Tania Sidana}\\
{\it Department of Mathematics, IIIT-Delhi}\\{\it New Delhi 110020, India}}
\date{}
\maketitle
\begin{abstract} Let $\mathcal{R}=\mathbb{F}_{p^m}[u]/\langle u^3 \rangle $ be the finite commutative chain ring  with unity, where $p$ is a  prime, $m$ is a positive integer and $\mathbb{F}_{p^m}$ is the finite field with $p^m$ elements. In this paper, we determine all repeated-root constacyclic codes of  arbitrary lengths over $\mathcal{R},$ their sizes and their dual codes. As an application, we  list some  isodual constacyclic codes  over $\mathcal{R}.$ We also determine Hamming distances, RT distances, and RT weight distributions of some repeated-root constacyclic codes  over $\mathcal{R}.$  \end{abstract}
\vspace{-4mm}\section{Introduction}\label{intro}\vspace{-2mm}
Berlekamp \cite{berl} first introduced and studied constacyclic codes over finite fields, which have a rich algebraic structure and are generalizations of  cyclic and negacyclic codes. Calderbank et al. \cite{calder}, Hammons et al.  \cite{hammons} and Nechaev \cite{nech} related binary non-linear codes  (e.g. Kerdock and Preparata codes) to linear codes over the finite commutative chain ring $\mathbb{Z}_{4}$ of integers modulo 4, with the help of a Gray map.  Since then, codes over finite commutative chain rings have received a great deal of attention.  However, their algebraic structures  are known only in a few cases. 

Towards this, Dinh and L$\acute{o}$pez-Permouth \cite{dinh} studied algebraic structures of simple-root cyclic and negacyclic codes of length $n$ over a finite commutative chain ring $R$ and their dual codes. In the same work,  they determined all negacyclic codes of length $2^t$ over the ring $\mathbb{Z}_{2^m}$ of integers modulo $2^m$ and their dual codes, where $t\geq 1$ and  $m \geq 2$ are  integers.  In a related work,  Batoul et al. \cite{bat}  proved that when $\lambda$ is an $n$th power of a unit in a finite chain ring $R,$  repeated-root $\lambda$-constacyclic  codes of length $n$ over $R$ are equivalent to cyclic codes.  Apart from this, many authors \cite{ashker,bannai,bonnecaze,huffman,udaya} investigated algebraic structures of linear and cyclic codes over the finite commutative  chain ring $\mathbb{F}_{2}[v]/\langle  v^2\rangle.$
 
 To describe the recent work, let $p$ be a prime, $s,m$ be positive integers, $\mathbb{F}_{p^m}$ be the finite field of order $p^m,$ and let $\mathbb{F}_{p^m}[v]/\langle  v^2\rangle$ be the finite commutative chain ring with unity.  Dinh \cite{dinh4} determined all constacyclic codes of length $p^s$ over $\mathbb{F}_{p^m}[v]/\langle  v^2\rangle$ and their Hamming distances. Later, Chen et al. \cite{chen}, Dinh et al. \cite{dinh5} and Liu et al. \cite{liu} determined   all constacyclic codes of length $2p^s$ over the ring $\mathbb{F}_{p^m}[v]/\langle  v^2\rangle,$ where $p$ is an odd prime.   Recently, Sharma and Rani \cite{sharma2} determined  all constacyclic codes of length $4p^s$ over $\mathbb{F}_{p^m}[v]/\langle  v^2\rangle$ and their dual codes,  where $p$ is an odd prime and $s,m$ are positive integers. Using a technique  different from that employed in  \cite{chen, dinh4, dinh5, liu, sharma2}, Cao et al. \cite{cao} determined  all $\alpha$-constacyclic codes of length $n p^s$ over $\mathbb{F}_{p^m}[v]/\langle  v^2\rangle$ and their dual codes by writing a canonical form decomposition for each code, where $\alpha $ is a non-zero element of $\mathbb{F}_{p^m}$ and  $n$ is a  positive integer with $\gcd(p,n)=1.$ In a recent work, Zhao et al. \cite{zhao} determined   all $(\alpha+\beta v)$-constacyclic codes of length $np^{s}$ over $\mathbb{F}_{p^m}[v]/ \langle  v^2\rangle$ and their dual codes, where $n$ is a positive integer coprime to $p,$  and $\alpha,\beta$ are non-zero elements of $\mathbb{F}_{p^m}.$ This completely solves the problem of determination of all constacyclic codes of length $np^s$ over $\mathbb{F}_{p^m}[v]/\langle  v^2\rangle$ and their dual codes,  where $n$ is a positive integer coprime to $p.$ In a subsequent work \cite{sharma3}, we determined all repeated-root constacyclic codes of arbitrary lengths over finite commutative chain rings with nilpotency index 2 and their dual codes. In the same work, we also listed some isodual repeated-root constacyclic codes and obtained Hamming distances, RT distances and RT weight distributions of some repeated-root constacyclic codes over finite commutative chain rings with nilpotency index 2.
  
  In a related work, Cao \cite{cao1} established algebraic structures of all $(1+ a w)$-constacyclic codes of arbitrary lengths over a finite commutative chain ring $R$ with the maximal ideal as $\langle w \rangle,$ where   $a$ is a unit in $R.$ 
   Later, Dinh et al. \cite{dinh6} studied repeated-root $(\alpha+a w)$-constacyclic codes of length $p^s$ over  a finite commutative chain ring $R$ with the maximal ideal as $\langle w \rangle,$ where $p$ is a prime number, $s \geq 1$ is an integer and $\alpha, a$ are units in $R.$  The results obtained in Dinh et al. \cite{dinh6} can also be obtained from the work of Cao \cite{cao1} and by establishing a ring isomorphism from $R[x]/\langle x^{p^s}-1-a \alpha^{-1}w \rangle$ onto $R[x]/\langle x^{p^s}-\alpha-a w\rangle$ as $A(x) \mapsto A(\alpha_0^{-1}x)$ for each $A(x) \in R[x]/\langle x^{p^s}-1-a \alpha^{-1}w \rangle,$ where $\alpha=\alpha_0^{p^s}$ (such an element $\alpha_0$ always exists in $\mathbb{F}_{p^m}$).    The constraint that $a$ is a unit in $R$ restricts their study to only a few special classes of repeated-root constacyclic codes  over $R.$ When $a$ is a unit in $R,$ the codes belonging to these special classes are direct sums of (principal) ideals of certain finite commutative chain rings. However,  when $a$ is a non-unit in $R,$ repeated-root constacyclic codes over $R$ can also be direct sums of  non-principal ideals. In another related work, Sobhani \cite{sobh} determined all $(\alpha+\gamma u^2)$-constacyclic codes of length $p^s$ over $\mathbb{F}_{p^m}[u]/\langle  u^3\rangle$ and their dual codes,  where $\alpha, \gamma$ are non-zero elements of $\mathbb{F}_{p^m}.$ 
  
   The main goal of this paper is to determine all repeated-root constacyclic codes of arbitrary lengths over $\mathbb{F}_{p^m}[u]/\langle u^3 \rangle,$ their sizes and  their dual codes, where $p$ is a prime and $m$ is a positive integer. The Hamming distances,  RT distances, and RT weight distributions are also determined for some repeated-root constacyclic codes  over $\mathbb{F}_{p^m}[u]/\langle u^3 \rangle.$ Some isodual repeated-root constacyclic codes  over $\mathbb{F}_{p^m}[u]/\langle u^3 \rangle$ are also listed.

 This paper is organized as follows: In Section \ref{prelim}, we state some basic definitions and  results  that are needed to derive our main results. In Section \ref{sec3}, we determine all repeated-root constacyclic codes of arbitrary lengths over $\mathbb{F}_{p^m}[u]/\langle u^3 \rangle,$ their dual codes and their sizes (Theorems \ref{t0}-\ref{t3}). As an application, we also determine some isodual repeated-root constacyclic codes over $\mathbb{F}_{p^m}[u]/\langle u^3 \rangle$ (Corollaries \ref{cor0}-\ref{cor2}).  In Section \ref{sec4}, we determine  Hamming distances, RT distances, and RT weight distributions of some repeated-root constacyclic codes over $\mathbb{F}_{p^m}[u]/\langle u^3 \rangle$ (Theorems \ref{HDt2}-\ref{RTW2}).  In Section \ref{con}, we mention a brief conclusion and discuss some interesting open problems in this direction.

\vspace{-4mm}\section{Some preliminaries}\label{prelim}\vspace{-2mm}
 A commutative ring   $R$ with unity is said to be (i) a local ring if it has a unique maximal ideal (consisting of all the non-units of $R$),    and (ii)  a chain ring if all the ideals of $R$ form a chain with respect to the  inclusion relation.  Then the following result is well-known. \vspace{-1mm}\begin{prop}\label{pr1}\cite{dinh} For a finite commutative ring $R$  with unity,  the following statements are equivalent:
\begin{enumerate}\vspace{-2mm}\item[(a)] $R$ is a local ring whose maximal ideal $M$ is principal,  i.e., $M=\langle  w\rangle $ for some $w \in R.$
\vspace{-2mm}\item[(b)] $R$ is a local principal ideal ring.
\vspace{-2mm}\item[(c)] $R$ is a chain ring  and all its ideals  are given by $\langle w^i\rangle ,$ $0 \leq i \leq e,$ where $e$ is the nilpotency index of $w.$ Furthermore, we have $|\langle w^i \rangle |=|R/\langle w\rangle |^{e-i}$ for $0 \leq i \leq e.$ (Throughout this paper, $|A|$ denotes the cardinality of the set $A.$)\vspace{-2mm}\end{enumerate}  \end{prop}

Now let $R$ be a finite commutative ring with unity and  let $N$ be a positive integer.  Let $R^N$ be the $R$-module consisting of all $N$-tuples over $R.$ For a unit $\lambda \in R,$ a $\lambda$-constacyclic code $\mathcal{C}$ of length $N$ over $R$ is defined as an $R$-submodule of $R^N$ satisfying the following property: $(a_0,a_1,\cdots,a_{N-1})\in \mathcal{C}$ implies that $(\lambda a_{N-1},a_0,a_1,\cdots,a_{N-2}) \in \mathcal{C}.$ The Hamming distance of $\mathcal{C},$ denoted by $d_H(\mathcal{C}),$ is defined as $d_H(\mathcal{C})=\min\{w_H(c): c \in \mathcal{C}\setminus \{0\}\},$ where  $w_H(c)$ is the number of non-zero components of $c$ and is called the Hamming weight of $c.$ The Rosenbloom-Tsfasman (RT) distance  of the code $\mathcal{C},$ denoted by $d_{RT}(\mathcal{C}),$ is defined as $d_{RT}(\mathcal{C})=\min\{ w_{RT}(c) | c\in \mathcal{C} \setminus \{ 0\}\},$ where  $w_{RT}(c)$ is the RT weight of $c$ and is defined as  \vspace{-2mm}\begin{equation*}w_{RT}(c)=\left\{\begin{array}{ll} 1+ \max \{ j | c_j \neq 0\} & \text{if } c=(c_0,c_1,\cdots,c_{N-1})  \neq 0; \\ 0 & \text{if }c=0.\end{array}\right.\vspace{-2mm}\end{equation*} The Rosenbloom-Tsfasman (RT)  weight distribution of $\mathcal{C}$ is defined as the list $\mathcal{A}_{0}, \mathcal{A}_{1},\cdots,\mathcal{A}_{N},$ where for $0 \leq \rho \leq N,$ $\mathcal{A}_{\rho}$ denotes the number of codewords in $\mathcal{C}$ having the RT weight as $\rho.$ The Hamming distance of a code is a measure of its error-detecting and error-correcting capabilities, while RT distances and RT weight distributions have applications in uniform distributions.

The dual code of $\mathcal{C},$ denoted by $\mathcal{C}^{\perp},$ is defined as $\mathcal{C}^{\perp}=\{u \in R^N: u.c=0 \text{ for all }c \in R^N\},$ where $u.c=u_0c_0+u_1c_1+\cdots + u_{N-1}c_{N-1}$ for $u=(u_0,u_1,\cdots,u_{N-1})$ and $c=(c_0,c_1,\cdots,c_{N-1})$ in $R^N.$ It is easy to observe that $\mathcal{C}^{\perp}$ is a $\lambda^{-1}$-constacyclic code of length $N$ over $R.$ The code $\mathcal{C}$ is said to be isodual if it is $R$-linearly equivalent to its dual code $\mathcal{C}^{\perp}.$
Under the standard $R$-module isomorphism $\psi: R^N \rightarrow R[x]/\langle x^N-\lambda\rangle ,$ defined as $\psi(a_0,a_1,\cdots,a_{N-1})=a_0+a_1 x+\cdots+a_{N-1}x^{N-1}+\langle x^N-\lambda \rangle$ for each $(a_0,a_1,\cdots,a_{N-1})\in R^N,$ the code $\mathcal{C}$ can be identified as an ideal of the ring $R[x]/\langle x^N-\lambda\rangle .$ Thus the study of $\lambda$-constacyclic codes of length $N$ over $R$ is equivalent to the study of ideals of the quotient ring $R[x]/\langle x^N-\lambda\rangle .$ From this point on, we shall represent elements of $R[x]/\langle x^N-\lambda \rangle$ by their representatives in $R[x]$ of degree less than $N,$ and we shall perform their addition and multiplication modulo $x^N-\lambda.$ Further, it is easy to see that the Hamming weight $w_H(c(x))$ of $c(x) \in R[x]/\langle x^N-\lambda \rangle$ is defined  as the number of non-zero coefficients of $c(x)$ and  the RT weight $w_{RT}(c(x))$ of $c(x) \in R[x]/\langle x^N-\lambda \rangle$ is defined as $w_{RT}(c(x))=\left\{\begin{array}{ll} 1+\text{deg }c(x) & \text{if }c(x) \neq 0;\\ 0 & \text{if }c(x)=0,\end{array}\right.$  (throughout this paper, $\text{deg }f(x)$ denotes the degree of a non-zero polynomial $f(x) \in R[x]$). The dual code $\mathcal{C}^{\perp}$ of $\mathcal{C}$  is given by $\mathcal{C}^{\perp}=\{u(x) \in  R[x]/\langle x^N-\lambda^{-1}\rangle : u(x)c^*(x)=0 \text{ in } R[x]/\langle x^N-\lambda^{-1}\rangle  \text{ for all } c(x) \in \mathcal{C}\},$ where $c^*(x)=x^{\text{deg }c(x)}c(x^{-1})$ for all $c(x) \in \mathcal{C}\setminus\{0\}$ and $c^*(x)=0$ if $c(x)=0.$  The annihilator of $\mathcal{C}$ is defined as $\text{ann}(\mathcal{C})=\{f(x) \in R[x]/\langle x^N-\lambda\rangle : f(x) c(x)=0 \text{ in }R[x]/\langle x^N-\lambda\rangle  \text{ for all } c(x) \in \mathcal{C}\}.$ One can easily observe that $\text{ann}(\mathcal{C})$ is an ideal of $R[x]/\langle x^N-\lambda\rangle .$ Further, for any ideal $I$ of $R[x]/\langle x^N-\lambda \rangle ,$ we define $I^*=\{f^*(x): f(x) \in I\},$ where $f^*(x)= x^{\text{deg }f(x)}f(x^{-1})$ if $f(x) \neq 0$ and $f^*(x)=0$ if $f(x)=0.$ It is easy to see that $I^*$ is an ideal of the ring $R[x]/\langle x^N-\lambda^{-1}\rangle .$ Now the following holds.                          
\vspace{-1mm}\begin{lem}\cite{chen} \label{pr2}If $\mathcal{C} \subseteq R[x]/\langle x^N-\lambda\rangle $ is a $\lambda$-constacyclic code of length $N$ over $R,$ then we have $\mathcal{C}^{\perp}=\text{ann}(\mathcal{C})^*.$ \end{lem}\vspace{-1mm}
From this point on, throughout this paper, let $R$ be the ring $\mathcal{R}=\mathbb{F}_{p^m}[u]/\langle u^3 \rangle.$ It is easy to observe that $\mathcal{R}=\mathbb{F}_{p^m}+u\mathbb{F}_{p^m}+u^2 \mathbb{F}_{p^m}$ with $u^3=0, $ and that any element $\lambda \in \mathcal{R}$ can be uniquely expressed as $\lambda=\alpha+u \beta  + u^2 \gamma ,$ where $\alpha,\beta, \gamma \in \mathbb{F}_{p^m}.$ Now we make the following observation.
\vspace{-1mm}\begin{lem} \cite{chen}\label{pr5} Let $\lambda  =\alpha+ u \beta  +u^2 \gamma  \in \mathcal{R},$ where $\alpha,\beta, \gamma \in \mathbb{F}_{p^m}.$  Then the following hold. \begin{enumerate}\vspace{-2mm}\item[(a)] $\lambda$ is a unit in $\mathcal{R}$ if and only if $\alpha \neq 0.$ \vspace{-2mm}\item[(b)]  There exists $\alpha_0 \in \mathbb{F}_{p^m}$ satisfying $\alpha_0^{p^s}=\alpha.$ \end{enumerate}\end{lem}\vspace{-1mm}

 The following theorem is useful in the determination of  Hamming distances of repeated-root constacyclic codes over $\mathcal{R}$ and is an extension of Theorem 3.4 of Dinh \cite{dinh4}.

\vspace{-1mm}\begin{thm}\label{dthm} For $\eta  \in \mathbb{F}_{p^m} \setminus \{0\},$ there exists $\eta_0 \in \mathbb{F}_{p^m}$ satisfying $\eta=\eta_0^{p^s}.$ Suppose that the polynomial $x^n-\eta_0$ is irreducible over $\mathbb{F}_{p^m}.$ Let $\mathcal{C}$ be an $\eta$-constacyclic code of length $np^s$ over $\mathbb{F}_{p^m}.$ Then we have $\mathcal{C}= \langle (x^n-\eta_0)^{\upsilon} \rangle ,$ where $0 \leq \upsilon \leq p^s.$ Moreover, the  Hamming distance $d_H(\mathcal{C})$ of the code $\mathcal{C}$ is given by 
\vspace{-2mm}\begin{equation*}d_H(\mathcal{C})=\left\{\begin{array}{ll}
1 & \text{ if } \upsilon=0;\\
\ell+2 & \text{ if } \ell p^{s-1}+1 \leq \upsilon \leq (\ell +1)p^{s-1} \text{ with } 0 \leq \ell \leq p-2;\\
(i+1)p^k & \text{ if } p^s-p^{s-k}+(i-1)p^{s-k-1}+1\leq \upsilon \leq p^s-p^{s-k}+ip^{s-k-1} \\ & \text{ with } 1\leq i \leq p-1 \text{ and } 1 \leq k \leq s-1;\\
0  & \text{ if } \upsilon=p^s.
\end{array}\right. \vspace{-2mm}\end{equation*}
\end{thm}
 \begin{proof} Working in a similar way as in Theorem 3.4 of Dinh \cite{dinh4}, the desired result follows.\vspace{-2mm}\end{proof}

  Next we proceed to study algebraic structures of all constacyclic codes of length $N=np^s$ over the ring $\mathcal{R}=\mathbb{F}_{p^m}+u\mathbb{F}_{p^m}+u^2 \mathbb{F}_{p^m},$ where $u^3=0,$ $p$ is a prime and $n, s,m$ are positive integers with $\gcd(n,p)=1.$
\vspace{-4mm} \section{Constacyclic codes of length $np^s$ over $\mathcal{R}$}\label{sec3}\vspace{-2mm}

Throughout this paper, let $p$ be a prime and let $n,s, m$ be positive integers with $\gcd(n,p)=1.$ Let $\mathbb{F}_{p^m}$ be the finite field of order $p^m,$ and let $\mathcal{R}=\mathbb{F}_{p^m}[u]/\langle u^3 \rangle$ be the finite commutative chain ring with unity. Let $\lambda=\alpha+\beta u +\gamma u^2,$ where $\alpha, \beta, \gamma \in \mathbb{F}_{p^m}$  and  $\alpha$ is non-zero. In this section, we willl determine all $\lambda$-constacyclic codes of length $np^s$ over $\mathcal{R}$ and their dual codes. We will  also determine the number of codewords in each code. Apart from this, we shall list some isodual constacyclic codes of length $np^s$ over $\mathcal{R}.$

To do this, we recall that a $\lambda$-constacyclic code of length $np^s$ over $\mathcal{R}$ is an ideal of the quotient ring $\mathcal{R}_{\lambda}=\mathcal{R}[x]/\langle x^{np^s}-\lambda\rangle .$ Further, by Lemma \ref{pr5}(b), there exists $\alpha_0 \in \mathbb{F}_{p^m}$ satisfying $\alpha_0^{p^s}=\alpha.$ Now let $x^n-\alpha_0=f_1(x)f_2(x)\cdots f_r(x)$ be the irreducible factorization of $x^n-\alpha_0$ over $\mathbb{F}_{p^m},$ where $f_1(x),f_2(x),\cdots,$ $f_r(x)$ are pairwise coprime monic (irreducible) polynomials over $\mathbb{F}_{p^m}.$  In the following lemma,  we factorize the polynomial $x^{np^s}-\lambda$ into pairwise coprime polynomials in $\mathcal{R}[x].$
\vspace{-1mm} \begin{lem} \label{fac} We have \begin{equation*}x^{np^s}-\lambda= \prod\limits_{j=1}^{r} \left(f_j(x)^{p^s}+ug_j(x)+u^2 h_j(x)\right),\end{equation*} where the polynomials  $g_1(x), g_2(x),\cdots, g_r(x), h_1(x), h_2(x),\cdots,h_r(x) \in \mathbb{F}_{p^m}[x]$ satisfy the following for $1 \leq j \leq r:$
  \begin{itemize} \vspace{-2mm} \item  $\gcd(f_j(x),g_j(x) ) =1$ when  $\beta \neq 0.$\vspace{-2mm}\item $g_j(x)=h_j(x)=0$ when $\beta=\gamma =0.$ \vspace{-2mm}\item  $g_j(x)=0$ and $\gcd(f_j(x),h_j(x))=1$ in $\mathbb{F}_{p^m}[x]$ when $\beta =0$ and $\gamma$ is non-zero. 
   \vspace{-2mm}\end{itemize}
  Moreover, the polynomials $f_1(x)^{p^s}+u g_1(x)+u^2 h_1(x), f_2(x)^{p^s}+u g_2(x)+u^2 h_2(x),\cdots,f_r(x)^{p^s}+u g_r(x)+u^2 h_r(x)$ are pairwise coprime in $\mathcal{R}[x].$
  \end{lem}
\begin{proof} To prove the result, we see that  \vspace{-2mm}\begin{equation}\label{poly} x^{np^s}-\lambda=(x^n-\alpha_0)^{p^s} -\beta u - \gamma u^2=f_1(x)^{p^s}f_2(x)^{p^s}\cdots f_r(x)^{p^s} -\beta u-u^2 \gamma.\vspace{-2mm}\end{equation}
Next we observe that for  $1 \leq j \leq r-1,$ the polynomials $f_j(x)^{p^s}$ and $\prod\limits_{i=j+1}^{r} f_{i}(x)^{p^s}$ are coprime in $\mathbb{F}_{p^m}[x],$  which implies that there exist $v_j(x), w_j(x) \in \mathbb{F}_{p^m}[x]$ satisfying $\text{deg }w_j(x) < \text{deg }f_j(x)^{p^s}$ and   \vspace{-2mm}\begin{equation}\label{gc} v_j(x)f_j(x)^{p^s}+w_j(x) \prod\limits_{i=j+1}^{r} f_{i}(x)^{p^s}=1.\vspace{-2mm}\end{equation}
Now by \eqref{poly} and \eqref{gc}, we obtain
 \vspace{-2mm}\begin{equation*}x^{np^s}-\lambda=\left\{f_1(x)^{p^s}-u \beta w_1(x)-u^2 w_1(x) \big(\gamma +\beta ^2 v_1(x)w_1(x)\big)\right\} \left\{
 \prod\limits_{i=2}^{r}f_i(x)^{p^s}-u\beta v_1(x) -u^2 v_1(x) \big(\gamma +\beta^2 v_1(x)w_1(x)\big)\right\}.\vspace{-2mm}\end{equation*} Further, using \eqref{gc} again,  we get  \vspace{-2mm}\begin{eqnarray*}\prod\limits_{i=2}^{r}f_i(x)^{p^s}-u\beta v_1(x) -u^2 v_1(x) \{\gamma +\beta^2 v_1(x)w_1(x)\}=\left\{f_2(x)^{p^s}-u\beta v_1(x)w_2(x)-u^2 v_1(x)w_2(x) \big(\gamma  +\beta^2v_1(x)w_1(x)\right. \\ \left. +\beta^2 v_1(x) v_2(x)w_2(x)\big)\right\} \left\{ \prod\limits_{i=3}^{r}f_i(x)^{p^s}-u \beta v_1(x)v_2(x)-u^2 v_1(x)v_2(x) \big(\gamma +\beta^2 v_1(x)w_1(x)+\beta^2 v_1(x)v_2(x)w_2(x)\big)\right\}.\vspace{-2mm}\end{eqnarray*} Proceeding like this, we obtain $x^{np^s}-\lambda =\prod\limits_{j=1}^{r}\Big(f_j(x)^{p^s}+ug_j(x)+u^2h_j(x) \Big)$ with  $g_1(x)=-\beta$ and $h_1(x)=-\gamma$ when $r=1;$ and  $g_{j}(x)=-\beta w_{j}(x) \prod\limits_{i=1}^{j-1} v_{i}(x)$ and $h_{j}(x)=-w_{j}(x)\prod\limits_{i=1}^{j-1}v_{i}(x) \Big( \gamma +\beta^2\sum\limits_{\ell=1}^{j}v_1(x)v_2(x)v_3(x)\cdots v_{\ell}(x)w_{\ell}(x) \Big)$ for $1 \leq j \leq r$ when $r \geq 2.$ From this, the desired result follows. \end{proof}\vspace{-2mm}

From now on, we define $k_j(x)=f_j(x)^{p^s}+ug_j(x)+u^2h_j(x)$ for $1 \leq j \leq r.$ Then we have  $x^{np^s}-\lambda =\prod\limits_{j=1}^{r}k_j(x).$ Further, if $\text{deg }f_j(x)=d_j,$ then we observe that $\text{deg }k_j(x)=d_jp^s$ for each $j.$ By Lemma \ref{fac}, we  see that  $k_1(x),k_2(x),\cdots,k_r(x)$ are pairwise coprime in $\mathcal{R}[x].$  This, by Chinese Remainder Theorem, implies that
  \vspace{-2mm}\begin{equation*}\label{decomk} \mathcal{R}_{\lambda}\simeq \bigoplus\limits_{j=1}^{r} \mathcal{K}_{j},\vspace{-2mm}\end{equation*} where $\mathcal{K}_{j}=\mathcal{R}[x]/\left<k_j(x)\right>$ for  $1 \leq j \leq r.$ Then  we observe the following:
 \begin{prop}\label{p1}\begin{enumerate}\vspace{-1mm}\item[(a)] Let $\mathcal{C}$ be a $\lambda$-constacyclic code of length $np^s$ over $\mathcal{R},$ i.e.,  an ideal of the ring $\mathcal{R}_{\lambda}.$ Then $\mathcal{C}=\mathcal{C}_1\oplus \mathcal{C}_2\oplus \cdots \oplus \mathcal{C}_r,$ where $\mathcal{C}_j$ is an ideal of $\mathcal{K}_{j}$ for $1 \leq j \leq r.$
\vspace{-2mm}\item[(b)] If $I_j$ is an ideal of $\mathcal{K}_{j}$ for $1 \leq j \leq r,$ then $I=I_1 \oplus I_2\oplus \cdots \oplus I_r$ is an ideal of $\mathcal{R}_{\lambda}$ (i.e., $I$ is a $\lambda$-constacyclic code of length $np^s$ over $\mathcal{R}$).  Moreover, we have $|I|=|I_1||I_2|\cdots |I_r|.$\vspace{-2mm}\end{enumerate}\end{prop}\vspace{-2mm} \begin{proof} Proof is trivial. \end{proof}\vspace{-1mm}

Next if $\mathcal{C}$ is a $\lambda$-constacyclic code of length $np^s$ over $\mathcal{R},$ then its dual code $\mathcal{C}^{\perp}$ is a $\lambda^{-1}$-constacyclic code of length $np^s$ over $\mathcal{R}.$  This implies that $\mathcal{C}^{\perp}$ is an ideal of the ring $\mathcal{R}_{\lambda^{-1}}=\mathcal{R}[x]/\langle x^{np^s}-\lambda^{-1}  \rangle.$ In order to determine $\mathcal{C}^{\perp}$ more explicitly,  we observe that $x^{np^s}-\lambda^{-1}=-\alpha^{-1}k_1^*(x)k_2^*(x)\cdots k_r^*(x).$ By applying Chinese Remainder Theorem again, we get $\mathcal{R}_{\lambda^{-1}}\simeq \bigoplus\limits_{j=1}^{r}\widehat{\mathcal{K}_{j}},$ where $\widehat{\mathcal{K}_{j}}=\mathcal{R}[x]/\langle k_j^*(x)\rangle $ for $1 \leq j \leq r.$ Then we have  the following:
\vspace{-1mm}  \begin{prop}\label{p2} Let $\mathcal{C}$ be a $\lambda$-constacyclic code of length $np^s$ over $\mathcal{R},$ i.e.,  an ideal of the ring $\mathcal{R}_{\lambda}.$ If $\mathcal{C}=\mathcal{C}_1\oplus \mathcal{C}_2\oplus \cdots \oplus \mathcal{C}_r$ with $\mathcal{C}_j$ an ideal of $\mathcal{K}_{j}$ for each $j,$  then the dual code $\mathcal{C}^{\perp}$ of $\mathcal{C}$ is given by $\mathcal{C}^{\perp}=\mathcal{C}_1^{\perp}\oplus \mathcal{C}_2^{\perp}\oplus \cdots \oplus \mathcal{C}_r^{\perp},$ where $\mathcal{C}_j^{\perp} =\{a_j(x) \in \widehat{\mathcal{K}_{j}}: a_j(x) c_j^*(x)=0 \text{ in }\widehat{\mathcal{K}_{j}} \text{ for all }c_j(x) \in \mathcal{C}_{j}\}$ is the orthogonal complement of  $\mathcal{C}_j$  for each $j.$ Furthermore, $\mathcal{C}_j^{\perp}$ is an ideal of $\widehat{\mathcal{K}_{j}}=\mathcal{R}[x]/\langle k_j^*(x)\rangle $ for each $j.$ \vspace{-2mm}\end{prop}
\vspace{-2mm}\begin{proof}   Its proof is straightforward. \end{proof}\vspace{-1mm}
In view of Propositions \ref{p1} and \ref{p2}, we see that to determine all $\lambda$-constacyclic codes of length $np^s$ over $\mathcal{R},$ their sizes  and their dual codes, we need to determine all ideals of the ring $\mathcal{K}_{j},$ their cardinalities and their orthogonal complements in $\widehat{\mathcal{K}_{j}}$ for $1 \leq j \leq r.$  To do so, throughout this paper, let $1 \leq j \leq r$ be a fixed integer. From now onwards, we shall represent elements of the rings $\mathcal{K}_{j}$ and $\widehat{\mathcal{K}_{j}}$ (resp. $\mathbb{F}_{p^m}[x]/\langle f_j(x)^{p^s}\rangle$) by their representatives in $\mathcal{R}[x]$ (resp. $\mathbb{F}_{p^m}[x]$) of degree less than $d_jp^s$ (resp. $d_jp^s$) and we shall perform their addition and multiplication modulo $k_j(x)$ and $k_j^*(x)$ (resp.  $f_j(x)^{p^s}$), respectively. 
To determine all ideals of the ring $\mathcal{K}_{j},$  we need to prove the following lemma.
\vspace{-1mm} \begin{lem} \label{nilred} Let $1 \leq j \leq r$ be fixed. In the ring $\mathcal{K}_{j},$ the following hold.
 \begin{enumerate}\vspace{-2mm}\item[(a)] Any non-zero polynomial $g(x) \in \mathbb{F}_{p^m}[x]$ satisfying $\gcd(g(x),f_j(x))=1$ is a unit in $\mathcal{K}_{j}.$ As a consequence, any non-zero polynomial in $\mathbb{F}_{p^m}[x]$ of degree less than $d_j$ is a unit in $\mathcal{K}_{j}.$ \vspace{-1mm}\item[(b)] $\left<f_j(x)^{p^s}\right>=\left\{\begin{array}{ll}\left<u\right> & \text{if }\beta \neq 0;\\\left<u^2\right> & \text{if }\beta=0 \text{ and } \gamma \neq 0; \\ \{0\} & \text{if } \beta=\gamma =0.\end{array}\right.$ 
 
 As a consequence, $f_j(x)$ is a nilpotent element of $\mathcal{K}_{j}$ with the nilpotency index as $3p^{s}$ when $\beta \neq 0,$ the nilpotency index of $f_j(x)$ is $2p^s$ when $\beta=0$ and $\gamma \neq 0,$ while the nilpotency index of $f_j(x)$ is $p^s$ when $\beta=\gamma=0.$
\vspace{-2mm}\end{enumerate}\vspace{-2mm}\end{lem}
\vspace{-2mm} \begin{proof} {\it (a)} As $f_j(x)$ is irreducible over $\mathbb{F}_{p^m}$ and  $\gcd(g(x),f_j(x))=1,$ we have $\gcd(g(x),f_j(x)^{p^s})=1$ in $\mathbb{F}_{p^m}[x],$ which implies that there exist polynomials $a(x),b(x) \in \mathbb{F}_{p^m}[x]$ such that $a(x)g(x)+b(x)f_j(x)^{p^s}=1.$ This implies that $a(x)g(x)+b(x) (f_j(x)^{p^s}+ug_j(x)+u^2h_j(x))=1+ub(x)(g_j(x)+uh_j(x)).$ From this, we get $a(x)g(x)=1+ub(x)(g_j(x)+uh_j(x))$ in $\mathcal{K}_{j}.$ As $u^3=0$ in $\mathcal{K}_{j},$ we see that  $1+ub(x)(g_j(x)+uh_j(x))$ is a unit in $\mathcal{K}_{j},$ which implies that $g(x)$ is a unit in $\mathcal{K}_{j}.$ 
 \\{\it (b)}  It follows immediately from Lemma \ref{fac} and part (a).
\vspace{-1mm} \end{proof}
Next for a positive integer $k,$ let $\mathcal{P}_{k}(\mathbb{F}_{p^m})=\{g(x) \in \mathbb{F}_{p^m}[x]: \text{ either }g(x)=0 \text{ or }\text{deg }g(x) < k\}.$ Note that every element $a(x) \in \mathcal{K}_j$ can be uniquely expressed as $a(x)=a_0(x)+ua_1(x)+u^2 a_2(x),$ where  $a_0(x),a_1(x),a_2(x) \in \mathcal{P}_{d_jp^s}(\mathbb{F}_{p^m}).$ Further, by repeatedly applying division algorithm in $\mathbb{F}_{p^m}[x],$ for $\ell \in \{0,1,2\},$ we can write $a_{\ell}(x)=\sum\limits_{i=0}^{p^s-1}  A_i^{(a_{\ell})}(x)f_j(x)^i,$  where $A_{i}^{(a_{\ell})}(x) \in \mathcal{P}_{d_j
}(\mathbb{F}_{p^m})$ for $0 \leq i \leq p^s-1.$  That is, each element $a(x) \in \mathcal{K}_{j}$ can be uniquely expressed as $a(x)=\sum\limits_{i=0}^{p^s-1}  A_i^{(a_0)}(x)f_j(x)^i+u \sum\limits_{i=0}^{p^s-1}  A_i^{(a_1)}(x)f_j(x)^i+u^2 \sum\limits_{i=0}^{p^s-1}  A_i^{(a_2)}(x)f_j(x)^i,$ where $A_i^{(a_{\ell})}(x) \in \mathcal{P}_{d_j}(\mathbb{F}_{p^m})$ for each $i$ and $\ell.$ 
Now to determine cardinalities of all ideals of $\mathcal{K}_{j},$ we  prove the following lemma.
\vspace{-1mm}\begin{lem} \label{card} Let $1 \leq j \leq r$ be a fixed integer. If $\mathcal{I}$ is an ideal of $\mathcal{K}_j,$ then  $\text{Res}_{u}(\mathcal{I})=\{a_0(x) \in \mathbb{F}_{p^m}[x]/\langle f_j(x)^{p^s}\rangle : a_0(x)+ua_1(x)+u^2 a_2(x) \in \mathcal{I} \text{ for some }a_1(x),a_2(x) \in \mathbb{F}_{p^m}[x]/\langle f_j(x)^{p^s}\rangle \},$ $\text{Tor}_{u}(\mathcal{I})=\{a_1(x) \in \mathbb{F}_{p^m}[x]/\langle f_j(x)^{p^s}\rangle : ua_1(x)+u^2 a_2(x) \in \mathcal{I} \text{ for some } a_2(x) \in \mathbb{F}_{p^m}[x]/\langle f_j(x)^{p^s}\rangle\}$ and $\text{Tor}_{u^2}(\mathcal{I})=\{a_2(x) \in \mathbb{F}_{p^m}[x]/\langle f_j(x)^{p^s}\rangle : u^2 a_2(x) \in \mathcal{I}\}$  are ideals of $\mathbb{F}_{p^m}[x]/\langle f_j(x)^{p^s}\rangle .$  Moreover, we have $|\mathcal{I}|=|Res_u(\mathcal{I})||\text{Tor}_{u}(\mathcal{I})| |\text{Tor}_{u^2}(\mathcal{I})|.$ 
\vspace{-2mm} \end{lem}
\vspace{-2mm}\begin{proof}  One can easily observe that  $\text{Res}_{u}(\mathcal{I}),$  $\text{Tor}_{u}(\mathcal{I})$ and $\text{Tor}_{u^2}(\mathcal{I})$  are  ideals of $\mathbb{F}_{p^m}[x]/\langle f_j(x)^{p^s}\rangle .$ In order to prove the second part,  we define a map   \begin{equation*}\phi : \mathcal{I} \rightarrow \text{Res}_{u}(\mathcal{I})\end{equation*} as $\phi(a(x))=a_0(x)$ for each $a(x)=a_0(x)+u a_1(x) +u^2 a_2(x)  \in \mathcal{I}$ with $a_0(x),a_1(x),a_2(x) \in \mathbb{F}_{p^m}[x]/\langle f_j(x)^{p^s}\rangle.$ We observe that $\phi$ is a surjective $\mathbb{F}_{p^m}[x]/\langle f_j(x)^{p^s}\rangle$-module homomorphism and its  kernel  is given by $K_{\mathcal{I}}=\{u a_1(x)  +u^2 a_2(x) \in \mathcal{I}: a_1(x),a_2(x) \in \mathbb{F}_{p^m}[x]/\langle f_j(x)^{p^s}\rangle\}.$ This implies that  \vspace{-2mm}\begin{equation}\label{cc}|\mathcal{I}|=|\text{Res}_{u}(\mathcal{I})| |K_{\mathcal{I}}|.\vspace{-2mm}\end{equation} We further define a map  \vspace{-2mm}\begin{equation*}\psi : K_{\mathcal{I}} \rightarrow \text{Tor}_{u}(\mathcal{I})\vspace{-1mm}\end{equation*} as $\psi(u a_1(x)+u^2 a_2(x))=a_1(x)$ for each $ua_1(x)+u^2 a_2(x) \in K_{\mathcal{I}},$ where $a_1(x),a_2(x) \in \mathbb{F}_{p^m}[x]/\langle f_j(x)^{p^s}\rangle.$ We see that $\psi$ is also a surjective $\mathbb{F}_{p^m}[x]/\langle f_j(x)^{p^s}\rangle$-module homomorphism with the kernel as $\text{ker } \psi =\{u^2 a_2(x) \in K_{\mathcal{I}}:  a_2(x) \in \mathbb{F}_{p^m}[x]/\langle f_j(x)^{p^s}\rangle\}.$ From this, it follows that  \vspace{-2mm}\begin{equation*}|K_{\mathcal{I}}|= |\text{Tor}_{u}(\mathcal{I})||\text{ker } \psi |=|\text{Tor}_{u}(\mathcal{I})||\text{Tor}_{u^2}(\mathcal{I})|,\end{equation*} which, by \eqref{cc}, implies that  \vspace{-2mm}\begin{equation*}|\mathcal{I}|=|\text{Res}_{u}(\mathcal{I})||\text{Tor}_{u}(\mathcal{I})||\text{Tor}_{u^2}(\mathcal{I})|.\vspace{-2mm}\end{equation*}
\vspace{-6mm}\end{proof}
To determine orthogonal complements of all ideals of $\mathcal{K}_{j},$ we  need  the following lemma.
\vspace{-1mm} \begin{lem} \label{lemdual} Let $1 \leq j \leq r $ be a fixed integer. Let $\mathcal{I}$ be an ideal of the ring $\mathcal{K}_{j}$ with the orthogonal complement as $\mathcal{I}^{\perp}.$ Then the following hold.
\begin{enumerate}\vspace{-2mm}\item[(a)] $\mathcal{I}^{\perp}$ is an ideal of $\widehat{\mathcal{K}_{j}}.$ \vspace{-2mm} \item[(b)] $\mathcal{I}^{\perp}=\{a^{*}(x)\in \widehat{\mathcal{K}_{j}}: a(x) \in \text{ann}(\mathcal{I})\}=\text{ann}(\mathcal{I})^*.$\vspace{-2mm}\item[(c)] If $\mathcal{I}=\langle f(x),ug(x),u^2 h(x)\rangle ,$ then we have $\mathcal{I}^{*}=\langle f^*(x),ug^*(x),u^2 h^*(x)\rangle .$
\vspace{-1mm}\item[(d)] For non-zero $f(x),g(x) \in \mathcal{K}_{j},$ let us define $(fg)(x)=f(x)g(x)$ and $(f+g)(x)=f(x)+g(x).$ If  $(fg)(x) \neq 0,$ then we have  $f^*(x) g^*(x)=x^{\text{deg }f(x)+\text{deg }g(x)-\text{deg }(fg)(x)}(fg)^*(x).$ If $(f+g)(x) \neq 0,$ then we have\vspace{2mm}\\
$(f+g)^*(x)=\left\{\begin{array}{ll} 
f^*(x)+x^{deg\, f(x)-deg\, g(x)}g^*(x)  & \text{if } \text{deg }f(x) > \text{deg } g(x);\\
x^{\text{deg }(f+g)(x)-\text{deg f(x)}}(f^*(x)+g^*(x)) & \text{if } \text{deg } f(x) = \text{deg }g(x).\end{array}\right.$
\vspace{-1mm}\end{enumerate}\end{lem}\begin{proof} Its proof is straightforward.\vspace{-2mm}\end{proof}

By the above lemma, we see that to determine $\mathcal{I}^{\perp},$ it is enough to determine $\text{ann}(\mathcal{I})$ for each ideal $\mathcal{I}$ of $\mathcal{K}_j.$ Further, to write down all ideals of $\mathcal{K}_{j},$ we see, by Lemma \ref{card}, that for each ideal  $\mathcal{I}$  of $\mathcal{K}_{j},$   $\text{Res}_{u}(\mathcal{I}),$ $\text{Tor}_{u}(\mathcal{I})$ and $\text{Tor}_{u^2}(\mathcal{I})$ all are ideals of the ring $\mathbb{F}_{p^m}[x]/\langle f_j(x)^{p^s}\rangle,$ which is a finite commutative chain ring with the maximal ideal as $\langle f_j(x)\rangle.$ Next by Proposition \ref{pr1}, we see that all the ideals of $\mathbb{F}_{p^m}[x]/\langle f_j(x)^{p^s}\rangle$ are given by $\langle f_j(x)^i \rangle$ with  $0 \leq i \leq p^s$ and that $|\langle {f_j(x)} ^{i}\rangle |=p^{md_j(p^s-i)}$ for each $i.$ This implies that $\text{Res}_{u}(\mathcal{I})=\langle f_j(x)^a \rangle,$ $\text{Tor}_{u}(\mathcal{I})=\langle f_j(x)^b\rangle $ and $\text{Tor}_{u^2}(\mathcal{I})=\langle f_j(x)^c\rangle$ for some integers $a,b,c$ satisfying $0 \leq c \leq b \leq a \leq p^s.$

First of all, we shall consider the case $\beta \neq 0.$ Here we see that when $\alpha_0=\mu^n$ for some $\mu \in \mathbb{F}_{p^m},$  each $\lambda$-constacyclic code of length $np^s$ over $\mathcal{R}$ is equivalent to a cyclic code of length $np^s$ over $\mathcal{R}$ and can be determined by using the results derived in Cao  \cite{cao1} via the map $\Psi: \mathcal{R}[x]/\langle x^{np^s}-1-\alpha^{-1}\beta u -\alpha^{-1}\gamma u^2 \rangle \rightarrow \mathcal{R}[x]/\langle x^{np^s}-\alpha-\beta u -\gamma u^2 \rangle,$ defined as $a(x) \mapsto a(\mu^{-1}x)$ for each $a(x) \in \mathcal{R}[x]/\langle x^{np^s}-1-\alpha^{-1}\beta u -\alpha^{-1}\gamma u^2 \rangle.$ However, when $\alpha_0$ (and hence $\alpha$) is not an $n$th power of an element in $\mathbb{F}_{p^m},$  this method can not be employed to determine all $(\alpha+\beta u +\gamma u^2)$-constacyclic codes of length $np^s$ over $\mathcal{R}.$ In fact, the problem of determination of all $(\alpha+\beta u +\gamma u^2)$-constacyclic codes of length $np^s$ over $\mathcal{R}$ and their dual codes is not yet completely solved.  Propositions \ref{p1} \&\ \ref{p2} and the following theorem completely solves this problem when $\beta$ is non-zero. \vspace{-1mm}\begin{thm}  \label{t0} When $\beta \neq 0,$ all ideals of the ring $\mathcal{K}_{j}$ are given by $\left< f_j(x)^{\ell}\right>,$ where $0 \leq \ell \leq 3p^s.$ Furthermore, for $0 \leq \ell \leq 3p^s,$ we have $|\left<f_j(x)^\ell \right>|=p^{m d_j (3p^s-\ell)}$ and $\text{ann}(\left<f_j(x)^{\ell}\right>)=\left<f_j(x)^{3p^s-\ell}\right>.$
\vspace{-2mm}\end{thm}
\begin{proof}  To prove this, we first observe that an element $a(x)\in \mathcal{K}_{j}$ can be  uniquely expressed as $a(x)=a_0(x)+ua_1(x)+u^2a_2(x)$, where $a_0(x),a_1(x),a_2(x)\in \mathcal{P}_{d_jp^s}(\mathbb{F}_{p^m}).$ By division algorithm in $\mathbb{F}_{p^m}[x],$ there exist unique polynomials  $q(x),r(x) \in \mathbb{F}_{p^m}[x]$ such that $a_0(x)={f_j(x)} q(x)+r(x)$, where either $r(x)=0$ or $\text {deg }r(x)< d_j.$ This implies that $a(x)={f_j(x)}q(x)+r(x)+ua_1(x)+u^2a_2(x).$ Now in view of Lemma \ref{nilred}(b), we see that  $a(x)$ is a unit in $\mathcal{K}_{j}$ if and only if $r(x)$ is a unit in $\mathcal{K}_{j}.$ Further, by Lemma \ref{nilred}(a), we see that $r(x) \in \mathbb{F}_{p^m}[x]$ is a unit in $\mathcal{K}_{j}$ if and only if $r(x) \neq 0.$ This shows that $a(x)$ is a non-unit in $\mathcal{K}_{j}$ if and only if $r(x)=0$ if and only if $a(x) \in \langle f_j(x)\rangle.$ That is, all the non-units of $\mathcal{K}_{j}$ are given by $\langle f_j(x)\rangle.$ Now using Proposition \ref{pr1} and Lemma \ref{nilred}(b), we see that $\mathcal{K}_{j}$ is a chain ring and all its ideals are given by $\langle f_j(x) ^{\ell}\rangle$ with  $0 \leq \ell \leq 3p^s.$   Furthermore, we observe that the residue field of $\mathcal{K}_{j}$ is given by $\overline{\mathcal{K}_j}=\mathcal{K}_j/\langle f_j(x)\rangle,$ and that  $|\overline{\mathcal{K}_j}|=p^{md_j}.$ Now using Proposition \ref{pr1} and Lemma \ref{nilred}(b) again, we obtain $| \langle f_j(x) ^{\ell}\rangle|=p^{md_j(3p^s-\ell)}$ for $0 \leq \ell \leq 3p^s.$  Further, it is easy to observe that $\text{ann}(\langle f_j(x)^{\ell} \rangle)=\langle f_j(x)^{3p^s-\ell}\rangle,$ which completes the proof of the theorem.\vspace{-1mm}\end{proof}
As a consequence of the above theorem, we deduce the following:
\vspace{-1mm}\begin{cor}\label{cor0}  Let $n \geq 1$ be an integer and $\alpha_0 \in \mathbb{F}_{p^m}$ be such that the binomial $x^n-\alpha_0$ is irreducible over $\mathbb{F}_{p^m}.$ Let $\alpha=\alpha_0^{p^s},$ and $\beta (\neq 0), \gamma \in \mathbb{F}_{p^m}.$   Then there exists an isodual $(\alpha+u\beta+u^2\gamma)$-constacyclic code of length $np^s$ over $\mathcal{R}$ if and only if $p=2.$ Moreover, when $p=2,$ the ideal $\langle (x^n-\alpha_0)^{ 3 \cdot 2^{s-1}}\rangle$ is the only isodual $(\alpha+u\beta+u^2 \gamma)$-constacyclic code of length $n2^s$ over $\mathcal{R}.$\vspace{-2mm}\end{cor}
\vspace{-2mm}\begin{proof}  On taking $f_j(x)=x^n-\alpha_0$ in Theorem \ref{t0},  we see that  all  $(\alpha+u\beta+u^2\gamma)$-constacyclic codes of length $np^s$ over $\mathcal{R}$ are given by $\left< (x^n-\alpha_0)^{\ell}\right>,$ where $0 \leq \ell \leq 3p^s.$ Furthermore, for $0 \leq \ell \leq 3p^s,$ the ideal $\left< (x^n-\alpha_0)^\ell \right>$ has $ p^{m n(3p^s-\ell)}$ elements and the annihilator  $\left< (x^n-\alpha_0)^{\ell}\right>$ is given by $\left< (x^n-\alpha_0)^{3p^s-\ell}\right>.$ Next we see that if a  code $\mathcal{C}=\left< (x^n-\alpha_0)^{\ell}\right>$ is isodual, then we must have $|\mathcal{C}|=|\mathcal{C}^{\perp}|.$ This gives $p^{mn(3p^s-\ell)}=p^{mn\ell}.$ This implies that  $3p^s=2\ell,$ which holds if and only if $p=2.$ So when $p$ is an odd prime, there does not exist any isodual $(\alpha+u\beta+u^2\gamma)$-constacyclic code of length $np^s$ over $\mathcal{R}.$ When $p=2,$ we get $\ell =3 \cdot 2^{s-1}.$   On the other hand, when $p=2,$ we observe that $\langle (x^n-\alpha_0)^{ 3 \cdot 2^{s-1}}\rangle$ is an isodual $(\alpha+u\beta+u^2 \gamma)$-constacyclic code of length $n2^s$ over $\mathcal{R},$ which completes the proof. 
 \end{proof}
\vspace{-1mm} \begin{rem}\label{R1} By Theorem 3.75 of \cite{lidl}, we see that the binomial $x^n-\alpha_0$ is irreducible over $\mathbb{F}_{p^m}$ if and only if the following two conditions are satisfied: (i) each prime divisor of $n$ divides the multiplicative order $e$ of $\alpha_0,$  but not $(p^m-1)/e $ and (ii) $p^m \equiv 1~(\text{mod }4)$ if $n \equiv 0~(\text{mod }4).$ 
 \end{rem}

 In the following theorem, we consider the case $\beta=\gamma=0,$ and we determine  all non-trivial ideals of the ring $\mathcal{K}_{j},$ their cardinalities and their annihilators. \begin{thm} \label{t2} Let $\beta=\gamma=0,$ and let  $\mathcal{I}$ be a non-trivial  ideal of the ring $\mathcal{K}_{j}$ with  $\text{Res}_{u}(\mathcal{I})=\langle f_j(x)^a \rangle,$ $\text{Tor}_{u}(\mathcal{I})=\langle f_j(x)^b\rangle $ and $\text{Tor}_{u^2}(\mathcal{I})=\langle f_j(x)^c\rangle$ for some integers $a,b,c$ satisfying $0 \leq c \leq b \leq a \leq p^s.$    Suppose that $B_i(x), C_k(x), Q_{\ell}(x), W_{e}(x)$ run over $\mathcal{P}_{d_j}(\mathbb{F}_{p^m})$ for each relevant $i, k , \ell$ and $e.$ Then the following hold.
 \begin{enumerate}\item[Type I:] When $a=b=p^s,$ we have  \begin{equation*}\mathcal{I}= \langle u^2 f_j(x)^c  \rangle ,\vspace{-2mm}\end{equation*} where $c  < p^s.$ Moreover, we have \begin{equation*}|\mathcal{I}|=p^{md_j(p^s-c)} \text{ and }\text{ann}(\mathcal{I})=\langle f_j(x)^{p^s-c}, u\rangle. \vspace{-2mm}\end{equation*}
 \item[Type II:] When $a=p^s$ and $b < p^s,$ we have we have  \begin{equation*}\mathcal{I}=\langle u f_j(x)^{b}+u^2f_j(x)^t G(x),u^2 f_j(x)^{c} \rangle ,\vspace{-2mm}\end{equation*}  where  $c+b-p^s \leq t < c $ if $G(x) \neq 0$  and  $G(x)$ is either 0 or  a unit in $\mathcal{K}_{j}$ of the form $\sum\limits_{i=0}^{c-t-1}B_{i}(x)f_j(x)^{i}.$ 
Moreover, we have \begin{equation*}|\mathcal{I}|=p^{md_j(2p^s-b-c)} \text{ and }\text{ann}(\mathcal{I})=\langle f_j(x)^{p^s-c}-u f_j(x)^{p^s-c+t-b}G(x), uf_j(x)^{p^s-b}, u^2 \rangle.\vspace{-2mm}\end{equation*}\item[Type III:] When $a < p^s,$  we have \begin{equation*}\mathcal{I}=\langle f_j(x)^a+u f_j(x)^{t_1} D_1(x)+u^2 f_j(x)^{t_2} D_2(x),u f_j(x)^{b}+u^2 f_j(x)^{\theta}V(x),u^2 f_j(x)^{c}\rangle ,\end{equation*} where   $a+b-p^s \leq t_1 < b $ if $D_1(x) \neq 0,$ $0 \leq t_2 < c$ if $D_2(x) \neq 0,$ $b+c-p^s \leq \theta< c$ if $V(x) \neq 0,$ $D_1(x)$ is either 0 or a unit  in $\mathcal{K}_{j}$ of the form $\sum\limits_{k=0}^{b-t_1-1}C_{k}(x)f_j(x)^{k},$ $D_2(x)$ is either 0 or a unit  in $\mathcal{K}_{j}$ of the form $\sum\limits_{\ell=0}^{c-t_2-1}Q_{\ell}(x)f_j(x)^{\ell}$ and  $V(x)$ is either 0 or a unit  in $\mathcal{K}_{j}$ of the form   $\sum\limits_{i=0}^{c-\theta-1}W_{i}(x)f_j(x)^i.$ Furthermore, we have $u^2\big(f_j(x)^{p^s-a+t_1-b+\theta}V(x)D_1(x)-f_j(x)^{p^s-a+t_2}D_2(x)\big) \in \langle u^2f_j(x)^{c}\rangle,$ i.e., there exists $A(x) \in \mathbb{F}_{p^m}[x]/\langle f_j(x)^{p^s}\rangle$ such that 
 $u^2\big( f_j(x)^{p^s-a+t_1-b+\theta}V(x)D_1(x)-f_j(x)^{p^s-a+t_2}D_2(x)\big)=u^2f_j(x)^{c}A(x).$ 
 
 Moreover, we have \begin{equation*}|\mathcal{I}|=p^{md_j(3p^s-a-b-c)}\vspace{-2mm}\end{equation*} and the annihilator of $\mathcal{I}$ is given by \vspace{2mm}\\$\text{ann}(\mathcal{I})=\langle f_j(x)^{p^s-c}-u f_j(x)^{p^s-c+\theta-b}V(x)+u^2 A(x), uf_j(x)^{p^s-b}-u^2 f_j(x)^{p^s-a+t_1-b}D_1(x), u^2 f_j(x)^{p^s-a}\rangle.$

\end{enumerate}\end{thm}
\begin{proof} As $\mathcal{I}$ is a non-trivial ideal of $\mathcal{K}_{j},$ we note that neither $a=0$ nor $a=b=c =p^s$ hold.  Further, by Lemma \ref{card}, we have  $|\mathcal{I}|=p^{md_j(3p^s-a-b-c)}.$ Now to write down all such non-trivial ideals of $\mathcal{K}_{j}$ and to determine their annihilators, we shall distinguish the following three cases: {\bf (i)} $a=p=p^s,$  {\bf (ii)} $a=p^s$ and $b < p^s,$  and {\bf (iii)} $a < p^s.$
\begin{description}\vspace{-2mm}\item[(i)] When $a=b=p^s,$ we have $\mathcal{I} \subseteq \langle u^2\rangle .$ In this case, we have  $0 \leq c< p^s.$ Here we observe that $\mathcal{I}=\langle u^2 f_j(x)^c \rangle.$ Now to find $\text{ann}(\mathcal{I}),$ we consider the ideal $\mathcal{B}_{1}=\langle f_j(x)^{p^s-c}, u, u^2 \rangle,$ and we see that $\mathcal{B}_{1} \subseteq \text{ann}(\mathcal{I})$ and that $|\mathcal{B}_{1} | =p^{md_j(2p^s+c)}.$ As
\begin{equation*} p^{md_j(p^s-c)} =|\mathcal{I}| =\frac{|\mathcal{K}_j|}{|\text{ann}(\mathcal{I}) |} \leq \frac{p^{3md_jp^s}}{|\mathcal{B}_{1} |}=p^{md_j(p^s-c)},\end{equation*}
we obtain $\text{ann}(\mathcal{I})=\mathcal{B}_{1}=\langle f_j(x)^{p^s-c}, u, u^2 \rangle. $

\vspace{-2mm}\item[(ii)] When $a=p^s$ and $b < p^s,$ we have  $\mathcal{I} \subseteq \langle u\rangle $ and $\mathcal{I} \not\subseteq \langle u^2\rangle .$  Here we observe that \vspace{-2mm}\begin{equation*}\mathcal{I} =\langle u{f_j(x)} ^{b}+u^2r(x),   u^2{f_j(x)} ^{c}\rangle\vspace{-2mm}\end{equation*} for some  $ r(x) \in \mathcal{K}_{j}.$ Let us write $u^2 r(x)=u^2 \sum\limits_{i=0}^{p^s-1}\mathcal{G}_i(x){f_j(x)} ^i,$ where $\mathcal{G}_i(x) \in \mathcal{P}_{d_j}(\mathbb{F}_{p^m})$ for $0 \leq i \leq p^s-1.$   Note that for all $i \geq c,$ we have $u^2 f_j(x)^i= u^2f_j(x) ^{c} f_j(x) ^{i-c} \in \mathcal{I},$ which implies that  $\mathcal{I}=<u f_j(x) ^{b}+u^2 \sum\limits_{i=0}^{c-1}\mathcal{G}_i(x) f_j(x) ^i, u^2f_j(x) ^{c}>.$  If $u^2\sum\limits_{i=0}^{c-1}\mathcal{G}_i(x) f_j(x) ^i \neq 0$ in $\mathcal{K}_{j},$ then choose the smallest integer $t~(0 \leq t < c)$ satisfying $\mathcal{G}_{t}(x) \neq 0,$ which gives $u^2\sum\limits_{i=0}^{c-1}\mathcal{G}_i(x) f_j(x) ^i=u^2f_j(x)^t G(x),$ where $G(x)=\sum\limits_{i=t}^{c-1} \mathcal{G}_{i}(x)f_j(x)^{i-t}$ is a unit in $\mathcal{K}_{j}.$ On the other hand, when $u^2\sum\limits_{i=0}^{c-1}\mathcal{G}_i(x) f_j(x) ^i=0$ in $\mathcal{K}_{j},$ let us choose $G(x)=0.$ From this, it follows that \vspace{-2mm}\begin{equation*}\mathcal{I} =\langle uf_j(x) ^{b}+u^2f_j(x) ^t G(x),   u^2f_j(x) ^{c}\rangle,\vspace{-2mm}\end{equation*} where $G(x)$ is either 0 or a unit in $\mathcal{K}_{j}$ of the form $\sum\limits_{i=0}^{c-t-1} a_{i}(x)f_j(x)^{i}$ with $a_{i}(x) \in \mathcal{P}_{d_j}(\mathbb{F}_{p^m})$ for $0 \leq i \leq c-t-1.$ 

Further, as $f_j(x)^{p^s-b}\{ uf_j(x) ^{b}+u^2f_j(x) ^t G(x) \} =   u^2 f_j(x)^{p^s-b+t} G(x) \in \mathcal{I}, $ we have $ p^s-b+t \geq c$ when $G(x) \neq 0.$

Moreover, let  $\mathcal{B}_{2}=\langle f_j(x)^{p^s-c}-u f_j(x)^{p^s-c+t-b}G(x), uf_j(x)^{p^s-b}, u^2 \rangle.$ We observe that $\mathcal{B}_{2} \subseteq \text{ann}(\mathcal{I})$ and $|\mathcal{B}_{2} |   \geq p^{md_j(p^s+b+c)}.$ 
Since \vspace{-2mm} \begin{equation*} p^{md_j(2p^s-b-c)} =|\mathcal{I}| =\frac{|\mathcal{K}_j|}{|\text{ann}(\mathcal{I}) |} \leq \frac{p^{3md_jp^s}}{|\mathcal{B}_{2} |} \leq p^{md_j(2p^s-b-c)},\vspace{-2mm}\end{equation*}
we obtain $|\text{ann}(\mathcal{I})|=|\mathcal{B}_{2} | = p^{md_j(p^s+b+c)}$ and $\text{ann}(\mathcal{I})=\mathcal{B}_{2}=\langle f_j(x)^{p^s-c}-u f_j(x)^{p^s-c+t-b}G(x), uf_j(x)^{p^s-b}, u^2 \rangle.$

\vspace{-2mm}\item[(iii)] When $a < p^s,$ we have $\mathcal{I} \not \subseteq \langle u \rangle .$ In this case, we see that   $ a >0.$ Here we observe that  \vspace{-2mm}\begin{equation*}\mathcal{I}= \langle f_j(x)^a+u r_1(x)+u^2r_2(x), uf_j(x)^b+u^2 q(x), u^2f_j(x)^{c}  \rangle\vspace{-2mm}\end{equation*} for some $r_1(x), r_2(x), q(x) \in \mathcal{K}_j.$ Let us write $u r_1(x)=u\sum\limits_{\ell=0}^{p^{s}-1}A_{\ell}(x)f_j(x)^{\ell},$ where $A_{\ell}(x) \in \mathcal{P}_{d_j}(\mathbb{F}_{p^m})$ for $0 \leq \ell \leq p^s-1.$ Now for all $\ell \geq b,$  we observe that $u{f_j(x)} ^{\ell}={f_j(x)} ^{\ell-b}\left\{uf_j(x)^b+u^2 q(x)\right\} -u^2 {f_j(x)} ^{\ell-b} q(x).$ This implies that \vspace{-2mm}\begin{equation*}\mathcal{I}= \langle f_j(x)^a + u \sum\limits_{\ell=0}^{b-1}A_\ell(x) f_j(x) ^i+u^2 \{r_2(x)- q(x) \sum\limits_{\ell=b}^{p^s-1}A_{\ell}(x)f_j(x) ^{\ell-a}\}, uf_j(x)^b+u^2 q(x), u^2f_j(x)^{c}  \rangle.\vspace{-2mm}\end{equation*} Next we write $u^2 \{r_2(x)- q(x) \sum\limits_{\ell=b}^{p^s-1}A_{\ell}(x)f_j(x) ^{\ell-a}\} =u^2\sum\limits_{k=0}^{p^s-1}B_{k}(x)f_j(x)^{k},$ where $B_{k}(x) \in \mathcal{P}_{d_j}(\mathbb{F}_{p^m})$ for $0 \leq k \leq p^s-1.$ Further, for all $k \geq c,$ we see that $u^2{f_j(x)} ^{k}=u^2f_j(x) ^{c}f_j(x)^{k-c} \in \mathcal{I},$ which implies that  \vspace{-2mm} \begin{equation*}\mathcal{I}=\langle f_j(x)^a+u \sum\limits_{\ell=0}^{b-1}A_{\ell}(x)f_j(x) ^{\ell}+u^2 \sum\limits_{k=0}^{c-1}B_{k}(x)f_j(x)^{k}, uf_j(x)^b+u^2 q(x), u^2f_j(x)^{c}  \rangle.\vspace{-2mm}\end{equation*} Next we write $u^2q(x)= u^2\sum\limits_{i=0}^{p^s-1}W_{i}(x)f_j(x)^i.$ We further observe that $\mathcal{I}=\langle f_j(x)^a+u \sum\limits_{\ell=0}^{b-1}A_{\ell}(x)f_j(x) ^{\ell}+u^2 \sum\limits_{k=0}^{c-1}B_{k}(x)f_j(x)^{k}, uf_j(x)^b+u^2 \sum\limits_{i=0}^{c-1}W_{i}(x)f_j(x)^i, u^2f_j(x)^{c}  \rangle.$ If $u\sum\limits_{\ell=0}^{b-1}A_{\ell}(x)f_j(x)^{\ell} \neq 0,$ then there exists a smallest integer $t_1$ satisfying $0 \leq t_1 < b$ and $A_{t_1}(x) \neq 0,$ and we can write  $u\sum\limits_{\ell=0}^{b-1}A_{\ell}(x)f_j(x)^{\ell}=uf_j(x)^{t_1}D_1(x),$ where $D_1(x)=\sum\limits_{\ell=t_1}^{b-1}A_{\ell}(x){f_j(x)} ^{\ell-t_1}$ is a unit in $\mathcal{K}_{j}.$ Moreover, if $u^2\sum\limits_{k=0}^{c-1}B_{k}(x)f_j(x)^{k}  \neq 0,$ then there exists a smallest integer $t_2$ satisfying $0 \leq t_2 < c$ and $B_{t_2}(x) \neq 0,$ and we can write $u^2\sum\limits_{k=0}^{c-1}B_{k}(x)f_j(x)^{k} =u^2f_j(x)^{t_2}D_2(x),$ where $D_2(x)=\sum\limits_{k=t_2}^{c-1}B_{k}(x)f_j(x)^{k-t_2} $ is a unit in $\mathcal{K}_{j}.$ Further, if  $u^2\sum\limits_{i=0}^{c-1}W_{i}(x)f_j(x)^i\neq 0,$ then there exists a smallest integer $\theta$ satisfying $0 \leq \theta < c$ and $W_{\theta}(x) \neq 0,$ and we can write $u^2\sum\limits_{i=0}^{c-1}W_{i}(x)f_j(x)^i =u^2f_j(x)^{\theta}V(x),$ where $V(x)=\sum\limits_{i=\theta}^{c-1}W_{i}(x)f_j(x)^{i-\theta}$ is a unit in $\mathcal{K}_{j}.$ From this, it follows that \vspace{-2mm}\begin{equation*}\mathcal{I}=\langle f_j(x)^a+u f_j(x) ^{t_1}D_1(x)+u^2 u f_j(x) ^{t_2}D_2(x), uf_j(x)^b+u^2 f_j(x)^\theta V(x), u^2f_j(x)^{c}  \rangle, \vspace{-2mm}\end{equation*} where $D_1(x)$ is either 0 or a unit in $\mathcal{K}_{j}$ of the form $\sum\limits_{\ell=t_1}^{b-1}A_{\ell}(x)f_j(x)^{\ell-t_1},$ $D_2(x)$ is either 0 or a unit in $\mathcal{K}_{j}$ of the form $\sum\limits_{k=t_2}^{c-1}B_{k}(x)f_j(x)^{k-t_2} $ and $V(x)$ is either 0 or a unit in $\mathcal{K}_{j}$ of the form $\sum\limits_{i=\theta}^{c-1}W_{i}(x)f_j(x)^{i-\theta}$ with $A_{\ell}(x),B_{k}(x),W_{i}(x) \in \mathcal{K}_{j}$ for each $\ell,k$ and $i.$

 In order to determine $\text{ann}(\mathcal{I}),$ we first observe that $uf_j(x)^{p^s-a+t_1}D_1(x)+u^2f_j(x)^{p^s-a+t_2}D_2(x) \in \mathcal{I},$ which implies that  $p^s-a+t_1 \geq b$ when $D_1(x) \neq 0.$ Next we see that $ f_j(x)^{p^s-b} \{ uf_j(x)^{b}+u^2f_j(x)^{\theta}V(x)\} \in \mathcal{I},$ which gives $p^s-b+\theta \geq c$ when $V(x) \neq 0.$ Moreover, as $uf_j(x)^{a}+u^2f_j(x)^{t_1}D_1(x) \in \mathcal{I}$ and $f_j(x)^{a-b} \{ uf_j(x)^{b}+u^2f_j(x)^{\theta}V(x)\} \in \mathcal{I},$ we note that $u^2\{ f_j(x)^{t_1}D_1(x)-f_j(x)^{a-b+\theta}V(x)\} \in \mathcal{I},$ which implies that $u^2\{ f_j(x)^{t_1}D_1(x)-f_j(x)^{a-b+\theta}V(x)\} \in \langle u^2 f_j(x)^{c} \rangle.$ From this, we obtain $u^2 f_j(x)^{p^s-c} \{ f_j(x)^{t_1}D_1(x)-f_j(x)^{a-b+\theta}V(x)\}=0.$

Further, we see that  $uf_j(x)^{p^s-a+t_1}D_1(x)+u^2f_j(x)^{p^s-a+t_2}D_2(x) \in \mathcal{I}$ can be rewritten as \\$f_j(x)^{p^s-a+t_1-b}D_1(x) \{uf_j(x)^{b}+u^2f_j(x)^{\theta}V(x)\}-u^2f_j(x)^{p^s-a+t_1-b+\theta}D_1(x)V(x)+u^2f_j(x)^{p^s-a+t_2}D_2(x),$ which implies that 
\vspace{-2mm}\begin{equation*}u^2 \{f_j(x)^{p^s-a+t_1-b+\theta}D_1(x)V(x)-f_j(x)^{p^s-a+t_2}D_2(x)\} \in \mathcal{I}.\vspace{-2mm}\end{equation*} This further implies that  
\vspace{-2mm}\begin{equation*}u^2 \{f_j(x)^{p^s-a+t_1-b+\theta}D_1(x)V(x)-f_j(x)^{p^s-a+t_2}D_2(x)\} \in \langle u^2f_j(x)^c \rangle.\vspace{-2mm}\end{equation*}
Let us write $u^2 \{f_j(x)^{p^s-a+t_1-b+\theta}D_1(x)V(x)-f_j(x)^{p^s-a+t_2}D_2(x)\} =u^2f_j(x)^c A(x),$ where $A(x) \in \mathbb{F}_{p^m}[x]/\langle f_j(x)^{p^s}\rangle.$
Now consider the ideal \\$\mathcal{B}_{3}=\langle f_j(x)^{p^s-c}-u f_j(x)^{p^s-c+\theta-b}V(x)+u^2 A(x), uf_j(x)^{p^s-b}-u^2 f_j(x)^{p^s-a+t_1-b}D_1(x), u^2 f_j(x)^{p^s-a}\rangle.$ Here we note  that $|\mathcal{B}_{3} |   \geq p^{md_j(a+b+c)}$ and  $\mathcal{B}_{3} \subseteq \text{ann}(\mathcal{I}).$ Further, as \vspace{-2mm}\begin{equation*} p^{md_j(3p^s-a-b-c)} =|\mathcal{I}| =\frac{|\mathcal{K}_j|}{|\text{ann}(\mathcal{I}) |} \leq \frac{p^{3md_jp^s}}{|\mathcal{B}_{3} |} \leq p^{md_j(3p^s-a-b-c)},\vspace{-2mm}\end{equation*}
we get $|\text{ann}(\mathcal{I})|=|\mathcal{B}_{3} | = p^{md_j(a+b+c)}$ and $\text{ann}(\mathcal{I})=\mathcal{B}_{3}.$ 
 \vspace{-2mm}\end{description}
This completes the proof of the theorem.\vspace{-2mm}\end{proof}

In the following corollary, we obtain some isodual $\alpha$-constacyclic codes of length $np^s$ over $\mathcal{R}$ when the binomial $x^n-\alpha_0$ is irreducible over $\mathbb{F}_{p^m}.$
\vspace{-2mm}\begin{cor} \label{cor1}  Let $n \geq 1$ be an integer and $\alpha_0 \in \mathbb{F}_{p^m}\setminus \{0\}$ be such that the binomial $x^n-\alpha_0$ is  irreducible  over $\mathbb{F}_{p^m}.$ Let $\alpha =\alpha_0^{p^s}\in \mathbb{F}_{p^m}.$  Following the same notations as in Theorem \ref{t2}, we have the following:
\begin{enumerate}
\vspace{-2mm}\item[(a)] There does not exist any  isodual $\alpha$-constacyclic code of Type I over $\mathcal{R}.$
\vspace{-2mm}\item[(b)] There exists an isodual $\alpha$-constacyclic code of Type II over $\mathcal{R}$ if and only if $p=2.$ In fact, when $p=2,$ the code $\langle u (x^n-\alpha_0)^{2^{s-1}}, u^2\rangle$ is the only isodual $\alpha$-constacyclic code of Type II over $\mathcal{R}.$ 
\vspace{-2mm}\item[(c)] There exists an isodual $\alpha$-constacyclic code of Type III over $\mathcal{R}$ if and only if $p=2.$ Moreover, when $p=2,$ the codes $\mathcal{C}=\langle (x^n-\alpha_0)^a+u^2 (x^n-\alpha_0)^{t_2} D_2(x),u (x^n-\alpha_0)^{2^{s-1}},u^2 (x^n-\alpha_0)^{2^s-a}\rangle,$  $2^{s-1} \leq a <2^s,$ are isodual $\alpha$-constacyclic codes of  Type III over $\mathcal{R}.$ 
\vspace{-2mm}\end{enumerate}
\vspace{-2mm}\end{cor}
\vspace{-2mm}\begin{proof} Let $\mathcal{C}$ be an $\alpha$-constacyclic code of length $np^s$ over $\mathcal{R}.$ For the code $\mathcal{C}$ to be isodual, we must have $|\mathcal{C}|=|\mathcal{C}^{\perp}|=|\text{ann}(\mathcal{C})|.$
\begin{enumerate}
\vspace{-2mm}\item[(a)] Let $\mathcal{C}$ be of Type I, i.e., $\mathcal{C}=\langle u^2 (x^n-\alpha_0)^c \rangle$ for some integer $c$ satisfying $0 \leq c  < p^s.$ By Theorem \ref{t2}, we see that $|\mathcal{C}|=p^{mn(p^s-c)}$ and $|\text{ann}(\mathcal{C})|=p^{mn(2p^s+c)}.$  Now if the code $\mathcal{C}$ is isodual, then we must have $|\mathcal{C}|=|\text{ann}(\mathcal{C})|.$ This implies that $p^s+2c=0,$ which is a contradiction. Hence  there does not exist any isodual $\alpha$-constacyclic code of Type I over $\mathcal{R}.$
\vspace{-2mm}\item[(b)] If the code $\mathcal{C}$ is of Type II, then $\mathcal{C} =\langle u (x^n-\alpha_0)^{b}+u^2(x^n-\alpha_0) ^t G(x),u^2 (x^n-\alpha_0)^{c} \rangle ,$ where $0 \leq c \leq b < p^s$ and $0 \leq t < c $ if $G(x) \neq 0.$ By Theorem \ref{t2}, we have $|\mathcal{C}|=p^{mn(2p^s-b-c)} ,$ $\text{ann}(\mathcal{C})=\langle (x^n-\alpha_0)^{p^s-c}-u (x^n-\alpha_0)^{p^s-c+t-b}G(x), u(x^n-\alpha_0)^{p^s-b}, u^2 \rangle$ and $|\text{ann}(\mathcal{C})|=p^{mn(p^s+b+c)}.$ Now if the code $\mathcal{C}$ is isodual, then we must have $|\mathcal{C}|=|\text{ann}(\mathcal{C})|,$ which gives $p=2$ and $c = 2^{s-1}-b.$ Further, if the code $\mathcal{C}$ is $\mathcal{R}$-linearly equivalent to $\text{ann}(\mathcal{C}),$ then $\text{Res}_{u}(\mathcal{C})=\{0\}$ must be $\mathbb{F}_{2^m}$-linearly equivalent to $\text{Res}_{u}(\text{ann}(\mathcal{C}))=\langle (x^n-\alpha_0)^{2^s-c}\rangle,$ which  implies that $c=0.$ This gives $b=2^{s-1}-c=2^{s-1}.$ 

On the other hand, when $p=2,$ $c=0$ and $b = 2^{s-1},$ by Theorem \ref{t2} again, we see that $\mathcal{C}=\text{ann}(\mathcal{C})$ holds, which implies that the codes $\mathcal{C}(\subseteq \mathcal{R}_{\alpha})$ and $\mathcal{C}^{\perp}(\subseteq \widehat{\mathcal{R}_{\alpha}})$ are $\mathcal{R}$-linearly equivalent. 

\vspace{-2mm}\item[(c)] If the code $\mathcal{C}$ is of Type III, then $\mathcal{C}=\langle (x^n-\alpha_0)^a+u (x^n-\alpha_0)^{t_1} D_1(x)+u^2 (x^n-\alpha_0)^{t_2} D_2(x),u (x^n-\alpha_0)^{b}+u^2 (x^n-\alpha_0)^{\theta}V(x),u^2 (x^n-\alpha_0)^{c}\rangle ,$ where $0  \leq c \leq b \leq a < p^s,$  $0 \leq t_1 < b $ if $D_1(x) \neq 0,$ $0 \leq t_2 < c$ if $D_2(x) \neq 0$ and $0 \leq \theta< c$ if $V(x) \neq 0.$

 Here by Theorem \ref{t2}, we have $|\mathcal{C}|=p^{mn(3p^s-a-b-c)}$ and $|\text{ann}(\mathcal{C})|=p^{mn(a+b+c)}.$ From this, we see that if  the code $\mathcal{C}$ is isodual, then we must have $3p^s=2(a+b+c),$ which implies that $p=2.$
 
 On the other hand, when $p=2,$ we see, by Theorem \ref{t2} again, that for $2^{s-1} \leq a < 2^s,$ the code $\mathcal{C}=\langle (x^n-\alpha_0)^a+u^2 (x^n-\alpha_0)^{t_2} D_2(x),u (x^n-\alpha_0)^{2^{s-1}},u^2 (x^n-\alpha_0)^{2^s-a}\rangle$ satisfies $\mathcal{C}=\text{ann}(\mathcal{C}),$ from which part (c) follows. 
\end{enumerate}
\vspace{-6mm}\end{proof}

In the following theorem, we consider the case $\beta=0$ and $\gamma \neq 0,$ and we determine all non-trivial ideals of the ring $\mathcal{K}_{j},$ their orthogonal complements and their cardinalities.
\vspace{-1mm}\begin{thm}\label{t3} Let $\beta=0$ and $\gamma$ be a non-zero element of $\mathbb{F}_{p^m}.$ Let  $\mathcal{I}$ be a non-trivial  ideal of the ring $\mathcal{K}_{j}$ with  $\text{Res}_{u}(\mathcal{I})=\langle f_j(x)^a \rangle,$ $\text{Tor}_{u}(\mathcal{I})=\langle f_j(x)^b\rangle $ and $\text{Tor}_{u^2}(\mathcal{I})=\langle f_j(x)^c\rangle$ for some integers $a,b,c$ satisfying $0 \leq c \leq b \leq a \leq p^s.$  Suppose that $B_i(x), C_k(x), Q_{\ell}(x), W_{e}(x)$ run over $\mathcal{P}_{d_j}(\mathbb{F}_{p^m})$ for each relevant $i, k , \ell$ and $e.$ Then the following hold.
 \begin{enumerate}\item[Type I:] When $a=b=p^s,$  we have  \begin{equation*}\mathcal{I}= \langle u^2 f_j(x)^c  \rangle ,\end{equation*} where $0 \leq c  < p^s.$ Moreover, we have \begin{equation*}|\mathcal{I}|=p^{md_j(p^s-c)} \text{ and }\text{ann}(\mathcal{I})=\langle f_j(x)^{p^s-c}, u\rangle. \end{equation*}
 \item[Type II:] When $a=p^s$ and $b < p^s,$  we have \begin{equation*}\mathcal{I}=\langle u f_j(x)^{b}+u^2f_j(x)^t G(x),u^2 f_j(x)^{c} \rangle ,\end{equation*}  where  $c+b-p^s \leq t < c $ if $G(x) \neq 0$  and  $G(x)$ is either 0 or  a unit in $\mathcal{K}_{j}$ of the form $\sum\limits_{i=0}^{c-t-1}B_{i}(x)f_j(x)^{i}.$ 
Moreover, we have \begin{equation*}|\mathcal{I}|=p^{md_j(2p^s-b-c)} \text{ and }\text{ann}(\mathcal{I})=\langle f_j(x)^{p^s-c}-u f_j(x)^{p^s-c+t-b}G(x), uf_j(x)^{p^s-b}, u^2 \rangle.\end{equation*}
\item[Type III:] When $a < p^s,$  we have \begin{equation*}\mathcal{I}=\langle f_j(x)^a+u f_j(x)^{t_1} D_1(x)+u^2 f_j(x)^{t_2} D_2(x),u f_j(x)^{b}+u^2 f_j(x)^{\theta}V(x),u^2 f_j(x)^{c}\rangle ,\end{equation*} where   $a+b-p^s \leq t_1 < b $ if $D_1(x) \neq 0,$ $0 \leq t_2 < c$ if $D_2(x) \neq 0,$ $b+c-p^s \leq \theta< c$ if $V(x) \neq 0,$ $D_1(x)$ is either 0 or a unit  in $\mathcal{K}_{j}$ of the form $\sum\limits_{k=0}^{b-t_1-1}C_{k}(x)f_j(x)^{k},$ $D_2(x)$ is either 0 or a unit  in $\mathcal{K}_{j}$ of the form $\sum\limits_{\ell=0}^{c-t_2-1}Q_{\ell}(x)f_j(x)^{\ell}$ and  $V(x)$ is either 0 or a unit  in $\mathcal{K}_{j}$ of the form   $\sum\limits_{i=0}^{c-\theta-1}W_{i}(x)f_j(x)^i.$ 
Furthermore, we have $u^2\big(h_j(x)+ f_j(x)^{p^s-a+t_1-b+\theta}V(x)D_1(x)-f_j(x)^{p^s-a+t_2}D_2(x)\big) \in \langle u^2f_j(x)^{c}\rangle,$ i.e., there exists $B(x) \in \mathbb{F}_{p^m}[x]/\langle f_j(x)^{p^s}\rangle$ such that 
 $u^2\big(h_j(x)+ f_j(x)^{p^s-a+t_1-b+\theta}V(x)D_1(x)-f_j(x)^{p^s-a+t_2}D_2(x)\big)=u^2f_j(x)^{c}B(x).$ 
 
 Moreover, we have  \begin{equation*}|\mathcal{I}|=p^{md_j(3p^s-a-b-c)}\end{equation*} and the annihilator of $\mathcal{I}$ is given by \vspace{2mm}\\$\text{ann}(\mathcal{I})=\langle f_j(x)^{p^s-c}-u f_j(x)^{p^s-c+\theta-b}V(x)+u^2 B(x), uf_j(x)^{p^s-b}-u^2 f_j(x)^{p^s-a+t_1-b}D_1(x), u^2 f_j(x)^{p^s-a}\rangle.$ 
\end{enumerate}\end{thm}
\begin{proof}
Working as in Theorem \ref{t2} and by applying Lemmas \ref{nilred}(c) and \ref{card}, the desired result follows.
\end{proof}
In the following corollary, we list some isodual $(\alpha+\gamma u^2)$-constacyclic codes of length $np^s$ over $\mathcal{R}$ when $\beta=0,$ $\gamma \neq 0$ and the binomial $x^n-\alpha_0$ is irreducible over $\mathbb{F}_{p^m}.$
\vspace{-1mm}\begin{cor}\label{cor2}  Let $n \geq 1$ be an integer and $\alpha_0 \in \mathbb{F}_{p^m}\setminus \{0\}$ be such that the binomial $x^n-\alpha_0$ is  irreducible  over $\mathbb{F}_{p^m}.$ Let $\alpha =\alpha_0^{p^s}\in \mathbb{F}_{p^m},$ and let  $\gamma$ be a non-zero element of $\mathbb{F}_{p^m}.$  Following the same notations as in Theorem \ref{t3}, we have the following:
\begin{enumerate}
\vspace{-2mm}\item[(a)] There does not exist any  isodual $(\alpha+ \gamma u^2)$-constacyclic code of Type I over $\mathcal{R}.$
\vspace{-2mm}\item[(b)] There exists an isodual $(\alpha+ \gamma u^2)$-constacyclic code of Type II over $\mathcal{R}$ if and only if $p=2.$ Furthermore, when $p=2,$ the code $\langle u (x^n-\alpha_0)^{2^{s-1}}, u^2\rangle$ is the only isodual $(\alpha + \gamma u^2)$-constacyclic code of Type II over $\mathcal{R}.$ 
\vspace{-2mm}\item[(c)] There exists an isodual $(\alpha+\gamma u^2)$-constacyclic code of Type III over $\mathcal{R}$ if and only if $p=2.$ Furthermore, when $p=2,$ the codes $\mathcal{C}=\langle (x^n-\alpha_0)^a+u (x^n-\alpha_0)^{a-2^{s-1}} \gamma ^{2^{m-1}}+u^2 (x^n-\alpha_0)^{t_2} D_2(x),u (x^n-\alpha_0)^{2^{s-1}}+u^2 \gamma ^{2^{m-1}}, u^2 (x^n-\alpha_0)^{2^s-a}\rangle,$  $2^{s-1} \leq a <2^s,$ are isodual $(\alpha + \gamma u^2)$-constacyclic codes of  Type III over $\mathcal{R}.$ 
\end{enumerate}
\end{cor}
\vspace{-3mm}\begin{proof}  Working in a similar manner as in Corollary \ref{cor1} and by applying Theorem \ref{t3}, the desired result follows.\end{proof}
\vspace{-5mm}\section{Hamming distances, RT distances and RT weight distributions}  \label{sec4}\vspace{-2mm}
Throughout this section, let  $n \geq 1$ be an integer and $\alpha_0 \in \mathbb{F}_{p^m}\setminus \{0\}$ be such that the binomial $x^n-\alpha_0$ is  irreducible  over $\mathbb{F}_{p^m}.$ Let  $\alpha =\alpha_0^{p^s}$ and $\beta, \gamma \in \mathbb{F}_{p^m}.$ When $\beta \neq 0,$ Sharma  \&\ Sidana \cite{sharma3} explicitly determined Hamming distances, RT distances and RT weight distributions of all repeated-root $(\alpha+\beta u +\gamma u^2)$-constacyclic codes over $\mathcal{R}.$ In this section, we shall consider the case $\beta=0,$ and we shall determine Hamming distances, RT distances and RT weight distributions of all $(\alpha+\gamma u^2)$-constacyclic codes of length $np^s$ over $\mathcal{R}.$  

 Now let $\mathcal{C}$ be an $(\alpha+\gamma u^2)$-constacyclic code of length $np^s$ over $\mathcal{R}.$ It is easy to see that $d_H(\mathcal{C})=d_{RT}(\mathcal{C})=0$ when $\mathcal{C}=0,$ while $d_H(\mathcal{C})=d_{RT}(\mathcal{C})=1$ when $\mathcal{C}=\langle 1\rangle.$   In the following theorem, we determine Hamming distances of all non-trivial $(\alpha+\gamma u^2)$-constacyclic codes of length $np^s$ over $\mathcal{R}.$ \vspace{-2mm}\begin{thm}\label{HDt2}  Let  $\mathcal{C}$ be a non-trivial $(\alpha+\gamma u^2)$-constacyclic code of length $np^s$ over $\mathcal{R}$ with $\text{Tor}_{u^2}(\mathcal{C})=\langle (x^n-\alpha_0)^c \rangle$ for some integer $c$ satisfying $0 \leq c < p^s$ (as determined  in Theorems \ref{t2} and \ref{t3}). Then the Hamming distance $d_H(\mathcal{C})$ of the code $\mathcal{C}$ is given by
 \begin{equation*}d_H(\mathcal{C})=\left\{\begin{array}{ll}
1 & \text{ if }  c=0;\\
\ell+2 & \text{ if } \ell p^{s-1}+1 \leq c \leq  (\ell +1)p^{s-1} \text{ with } 0 \leq \ell \leq p-2;\\
(i+1)p^k & \text{ if } p^s-p^{s-k}+(i-1)p^{s-k-1}+1\leq c \leq p^s-p^{s-k}+ip^{s-k-1} \\ & \text{ with } 1\leq i \leq p-1 \text{ and } 1 \leq k \leq s-1.
\end{array}\right. \end{equation*}
\end{thm}
\vspace{-1mm}\begin{proof}  To prove the result,  we assert that \begin{equation}\label{as} d_H(\mathcal{C})=d_H(\text{Tor}_{u^2}(\mathcal{C})).\end{equation} 
 
 To prove this assertion, we note that  $ \langle u^2 (x^{n}-\alpha_0)^{c} \rangle \subseteq \mathcal{C},$ which implies that \begin{equation}\label{ob1} d_H(\langle u^2 (x^{n}-\alpha_0)^{c}\rangle) \geq d_H(\mathcal{C}).\end{equation}\vspace{-2mm}  Next we observe  that  
\vspace{-1mm} \begin{equation}\label{Q} w_H(Q(x)) \geq w_H(uQ(x))  \text{ for each }Q(x) \in \mathcal{R}_{\alpha+\gamma u^2}.\end{equation}
  When $\mathcal{C}$ is of Type I, we have $\mathcal{C}=\langle u^2 (x^n-\alpha_0)^{c}\rangle.$ Here we have  $d_H(\mathcal{C})=d_H(\langle u^2 (x^{n}-\alpha_0)^{c} \rangle).$

When $\mathcal{C}$ is of Type II, we have $\mathcal{C}=\langle u (x^{n}-\alpha_0)^{b}+u^2(x^{n}-\alpha_0)^t G(x),u^2 (x^{n}-\alpha_0)^{c} \rangle,$ where $c \leq b <  p^s,$ $c+b-p^s \leq t < c $ if $G(x) \neq 0$  and  $G(x)$ is either 0 or  a unit in $\mathbb{F}_{p^m}[x]/\langle f_j(x)^{p^s}\rangle.$ Here for each codeword $Q(x) \in \mathcal{C} \setminus \langle u^2 (x^{n}-\alpha_0)^{c} \rangle,$  we see, by \eqref{Q}, that $w_H(Q(x)) \geq w_H(uQ(x)) \geq  d_H(\langle u^2 (x^{n}-\alpha_0)^{c}\rangle ).$ From this, we obtain $d_H(\mathcal{C}) \geq d_H(\langle u^2 (x^{n}-\alpha_0)^{c} \rangle).$ 

When $\mathcal{C}$ is of  Type III, we have $\mathcal{C}=\langle (x^{n}-\alpha_0)^a+u (x^{n}-\alpha_0)^{t_1} D_1(x)+u^2 (x^{n}-\alpha_0)^{t_2} D_2(x),u (x^{n}-\alpha_0)^{b}+u^2(x^{n}-\alpha_0)^{\theta}V(x),u^2 (x^{n}-\alpha_0)^{c}\rangle,$ where  $c \leq b \leq a <p^s,$ $a+b-p^s \leq t_1 < b $ if $D_1(x) \neq 0,$ $0 \leq t_2 < c$ if $D_2(x) \neq 0,$ $b+c-p^s \leq \theta< c$ if $V(x) \neq 0$ and $D_1(x),D_2(x),V(x)$ are either 0 or a units  in $\mathbb{F}_{p^m}[x]/\langle f_j(x)^{p^s}\rangle.$  Here for each codeword $Q(x) \in \mathcal{C} \setminus \langle u^2 (x^{n}-\alpha_0)^{c} \rangle$ and $Q(x) \in \langle u \rangle,$ by \eqref{Q}, we see that $w_H(Q(x)) \geq w_H(uQ(x)) \geq  d_H(\langle u^2 (x^{n}-\alpha_0)^{c}\rangle ).$ Further, for a codeword $Q(x) \in \mathcal{C} \setminus \langle u \rangle,$ by \eqref{Q} again,   we note that $w_H(Q(x)) \geq w_H(u^2 Q(x)) \geq  d_H(\langle u^2 (x^{n}-\alpha_0)^{c}\rangle ).$ This implies that $d_H(\mathcal{C}) \geq d_H(\langle u^2 (x^{n}-\alpha_0)^{c} \rangle).$ 

From this and by \eqref{ob1}, we get $d_H(\mathcal{C}) = d_H(\langle u^2 (x^{n}-\alpha_0)^{c}\rangle).$ Further, we observe that $\label{ob}d_H(\langle u^2 (x^n-\alpha_0)^{c}\rangle)=d_H(\text{Tor}_{u^2}(\mathcal{C})),$ from which the assertion \eqref{as} follows.
Now by applying Theorem \ref{dthm}, we get the desired result.\end{proof}

 In the following theorem, we determine RT distances of all non-trivial $(\alpha+\gamma u^2)$-constacyclic codes of length $np^s$ over $\mathcal{R}.$ 
\vspace{-2mm}\begin{thm}   \label{RTD2} Let  $\mathcal{C}$ be a non-trivial $(\alpha+\gamma u^2)$-constacyclic code of length $np^s$ over $\mathcal{R}$ with $\text{Tor}_{u^2}(\mathcal{C})=\langle (x^n-\alpha_0)^c \rangle$ for some integer $c$ satisfying $0 \leq c < p^s$ (as determined  in Theorems \ref{t2} and \ref{t3}). Then  the RT distance $d_{RT}(\mathcal{C})$ of the code $\mathcal{C}$ is given by \begin{equation*}d_{RT}(\mathcal{C})=nc+1.\end{equation*}
\end{thm}
\vspace{-2mm}\begin{proof}   To prove the result, we first observe that\vspace{-1mm}\begin{equation}\label{eq1} w_{RT}(Q(x)) \geq w_{RT}(uQ(x)) \text{ for each }Q(x) \in \mathcal{R}_{\alpha + \gamma u^2}.\end{equation}\vspace{-3mm}
\begin{description}\vspace{-5mm}\item[(i)] When $\mathcal{C}$ is of Type I, we have $\mathcal{C}=\langle u^2 (x^{n}-\alpha_0)^{c}\rangle.$  Here we note that $\mathcal{C}=\langle u^2 (x^{n}-\alpha_0)^{c}\rangle =\{ u^2 (x^{n}-\alpha_0)^{c}f(x) \  | \  f(x) \in \mathbb{F}_{p^m}[x]\}.$ Now for each non-zero $Q(x) \in \mathcal{C},$ by \eqref{eq1}, we see that $w_{RT}(Q(x)) \geq w_{RT}(u^2 (x^{n}-\alpha_0)^{c})=nc+1,$ which implies that $d_{RT}(\mathcal{C}) \geq nc+1.$ Since $u^2 (x^{n}-\alpha_0)^{c} \in \mathcal{C},$ we obtain $d_{RT}(\mathcal{C}) = nc+1.$ 
\vspace{-1mm}\item[(ii)] When $\mathcal{C}$ is of Type II, we have $\mathcal{C}=\langle u (x^{n}-\alpha_0)^{b}+u^2(x^{n}-\alpha_0)^t G(x),u^2 (x^{n}-\alpha_0)^{c} \rangle ,$ where $c \leq b <p^s,$ $c+b-p^s \leq t < c $ if $G(x) \neq 0$  and  $G(x)$ is either 0 or  a unit in $\mathbb{F}_{p^m}[x]/\langle f_j(x)^{p^s}\rangle.$  Here by \eqref{eq1},  we note that  $w_{RT}(Q(x)) \geq w_{RT}(u Q(x))$ for each $Q(x) \in \mathcal{C} \setminus \langle u^2 \rangle,$ which implies that $w_{RT}(Q(x)) \geq d_{RT}(\langle u^2 (x^{n}-\alpha_0)^{c}\rangle) $ for each $Q(x) \in \mathcal{C} \setminus \langle u^2 \rangle.$ From this, we get $d_{RT}( \mathcal{C}) \geq d_{RT}(\langle u^2 (x^{n}-\alpha_0)^{c} \rangle).$ Since $\langle u^2 (x^{n}-\alpha_0)^{c}  \rangle \subseteq \mathcal{C},$ we have $d_{RT}(\langle u^2 (x^{n}-\alpha_0)^{c} \rangle) \geq d_{RT}( \mathcal{C}).$ This implies that $d_{RT}( \mathcal{C}) = d_{RT}(\langle u^2 (x^{n}-\alpha_0)^{c} \rangle).$ From this and by  case (i),  we get  $d_{RT}(\mathcal{C})=nc+1.$
\vspace{-1mm}\item[(iii)] When $\mathcal{C}$ is of Type III, we have $\mathcal{C}=\langle (x^{n}-\alpha_0)^a+u (x^{n}-\alpha_0)^{t_1} D_1(x)+u^2 (x^{n}-\alpha_0)^{t_2} D_2(x),u (x^{n}-\alpha_0)^{b}+u^2(x^{n}-\alpha_0)^{\theta}V(x), u^2 (x^{n}-\alpha_0)^{c}\rangle,$  where   $c \leq b \leq a <p^s,$ $a+b-p^s \leq t_1 < b $ if $D_1(x) \neq 0,$ $0 \leq t_2 < c$ if $D_2(x) \neq 0,$ $b+c-p^s \leq \theta< c$ if $V(x) \neq 0$ and $D_1(x),D_2(x),V(x)$ are either 0 or a units  in $\mathbb{F}_{p^m}[x]/\langle f_j(x)^{p^s}\rangle.$ For each $Q(x) \in \mathcal{C} \setminus \langle u \rangle,$ by \eqref{eq1}, we see that $w_{RT}(Q(x)) \geq w_{RT}(u^2 Q(x)).$ From this, we get $w_{RT}(Q(x)) \geq d_{RT}(\langle u^2 (x^{n}-\alpha_0)^{c}\rangle) $ for each $Q(x) \in \mathcal{C} \setminus \langle u \rangle.$ Further, for a codeword $Q(x) \in \mathcal{C} \setminus \langle u^2 (x^{n}-\alpha_0)^{c} \rangle$ with $Q(x) \in \langle u \rangle,$  by \eqref{eq1} again, we see that $w_{RT}(Q(x)) \geq w_{RT}(uQ(x)) \geq  d_{RT}(\langle u^2 (x^{n}-\alpha_0)^{c}\rangle ).$ This implies that $d_{RT}( \mathcal{C}) \geq d_{RT}(\langle u^2 (x^{n}-\alpha_0)^{c}\rangle).$ On the other hand, as $\langle u^2 (x^{n}-\alpha_0)^{c} \rangle \subseteq \mathcal{C},$ we have $d_{RT}(u^2 (x^{n}-\alpha_0)^{c}\rangle) \geq d_{RT}(\mathcal{C}),$ which implies that $d_{RT}(\mathcal{C})=d_{RT}(\langle u^2(x^{n}-\alpha_0)^{c}\rangle).$ From this and by case (i), we get  $d_{RT}(\mathcal{C})=nc+1.$ \end{description}

This completes the proof of the theorem.\vspace{-3mm}\end{proof}

In the following theorem, we determine RT weight distributions of all  $(\alpha+\gamma u^2)$-constacyclic codes of length $np^s$ over $\mathcal{R}.$
\vspace{-2mm}\begin{thm} \label{RTW2} Let  $\mathcal{C}$ be an $(\alpha+\gamma u^2)$-constacyclic code of length $np^s$ over $\mathcal{R}$ with  $\text{Res}_{u}(\mathcal{C})=\langle (x^n-\alpha_0)^a \rangle,$ $\text{Tor}_{u}(\mathcal{C})=\langle (x^n-\alpha_0)^b\rangle $ and $\text{Tor}_{u^2}(\mathcal{C})=\langle (x^n-\alpha_0)^c\rangle$ for some integers $a,b,c$ satisfying $0 \leq c \leq b \leq a \leq p^s$   (as determined  in Theorems \ref{t2} and \ref{t3}). For $0 \leq \rho \leq np^s,$ let $\mathcal{A}_{\rho}$ denote the number of codewords in $\mathcal{C}$ having the RT weight as $\rho.$
\begin{enumerate}
\vspace{-2mm}\item[(a)] If $\mathcal{C}=\{0\},$ then we have $\mathcal{A}_0=1$ and $\mathcal{A}_\rho=0$ for $1 \leq \rho \leq np^s.$ \vspace{-2mm}\item[(b)] If $\mathcal{C}=\langle 1 \rangle,$ then we have $\mathcal{A}_0=1$ and $\mathcal{A}_\rho=(p^{3m}-1)p^{3m(\rho-1)}$ for $1 \leq \rho \leq np^s.$
\vspace{-2mm}\item[(c)] If $\mathcal{C}=\langle u^2(x^{n}-\alpha_0) ^{c}\rangle $ is of Type I, then we have\vspace{-1mm} \begin{equation*}\mathcal{A}_\rho=\left\{\begin{array}{ll}
1 & \text{ if } \rho=0;\\
0 & \text{ if } 1\leq \rho \leq nc;\\
(p^{m}-1)p^{m(\rho-nc-1)}  & \text{ if } nc+1 \leq \rho \leq np^s.
\end{array}\right. \vspace{-1mm}\end{equation*}

\vspace{-2mm}\item[(d)] If $\mathcal{C}=\langle u (x^{n}-\alpha_0)^{b}+u^2(x^{n}-\alpha_0)^t G(x),u^2 (x^{n}-\alpha_0)^{c} \rangle $ is of Type II, then we have \vspace{-1mm}\begin{equation*}\mathcal{A}_\rho=\left\{\begin{array}{ll}
1 & \text{ if } \rho=0;\\
0 & \text{ if } 1\leq \rho \leq nc;\\
(p^{m}-1)p^{m(\rho-nc-1)}  & \text{ if } nc+1 \leq \rho \leq nb;\\
(p^{2m}-1)p^{m(2\rho-nb-nc-2)} & \text{ if } nb+1 \leq \rho \leq np^s. \end{array}\right. \vspace{-1mm}\end{equation*}

\vspace{-2mm}\item[(e)] If $\mathcal{C}=\langle (x^{n}-\alpha_0)^a+u (x^{n}-\alpha_0)^{t_1} D_1(x)+u^2 (x^{n}-\alpha_0)^{t_2} D_2(x),u (x^{n}-\alpha_0)^{b}+u^2(x^{n}-\alpha_0)^{\theta}V(x),u^2 (x^{n}-\alpha_0)^{c}\rangle $ is of Type III, then we have \vspace{-1mm}\begin{equation*}\mathcal{A}_\rho=\left\{\begin{array}{ll}
1 & \text{ if } \rho=0;\\
0 & \text{ if } 1\leq \rho \leq nc;\\
(p^{m}-1)p^{m(\rho-nc-1)}  & \text{ if } nc+1 \leq \rho \leq nb;\\
(p^{2m}-1)p^{m(2\rho-nb-nc-2)} & \text{ if } nb+1 \leq \rho \leq na;\\
(p^{3m}-1)p^{m(3\rho-na-nb-nc-3)} & \text{ if } na+1 \leq \rho \leq np^s. \end{array}\right. \vspace{-1mm}\end{equation*}
\end{enumerate}
\end{thm}

\begin{proof} Proofs of parts (a) and (b) are trivial. To prove parts (c)-(e), by Theorem \ref{RTD2}(c), we see that $d_{RT}(\mathcal{C})=nc+1,$ which implies that $\mathcal{A}_\rho=0$ for $1 \leq \rho \leq nc.$ So from now on, we assume that $nc+1 \leq \rho \leq np^s.$

To prove (c), let $\mathcal{C}= \langle u^2(x^{n}-\alpha_0) ^{c}\rangle .$ Here we see that $\mathcal{C}= \langle u^2(x^{n}-\alpha_0) ^{c}\rangle =\{u^2 (x^{n}-\alpha_0)^{c} F(x)\ |\ F(x) \in \mathbb{F}_{p^m}[x]\}.$  This implies that the codeword $u^2(x^{n}-\alpha_0)^{c} F(x) \in \mathcal{C}$ has RT weight $\rho$ if and only if   $\text{ deg } F(x)=\rho-nc-1.$ From this, we obtain $\mathcal{A}_\rho=(p^{m}-1)p^{m(\rho-nc-1)} .$

To prove (d), let $\mathcal{C}=\langle u (x^{n}-\alpha_0)^{b}+u^2(x^{n}-\alpha_0)^t G(x),u^2 (x^{n}-\alpha_0)^{c} \rangle .$  Here we observe that each codeword $Q(x)\in \mathcal{C}$ can be uniquely expressed as $Q(x)= (u (x^{n}-\alpha_0)^{b}+u^2(x^{n}-\alpha_0)^t G(x))A_Q(x)+u^2 (x^{n}-\alpha_0)^{c} B_Q(x),$ where $A_Q(x),B_Q(x) \in \mathbb{F}_{p^m}[x]$ satisfy  $ \text{ deg }A_Q(x) \leq n(p^s-b)-1$ if $A_Q(x) \neq 0$ and $ \text{ deg }B_Q(x) \leq n(p^s-c)-1$ if $B_Q(x) \neq 0.$   From this, we see that if $ nc+1 \leq \rho \leq nb,$ then the RT weight of the codeword $Q(x) \in \mathcal{C}$ is $\rho$ if and only if  $A_Q(x)=0$ and $\text{ deg }B_Q(x)=\rho-nc-1.$ This implies that $\mathcal{A}_\rho=(p^{m}-1)p^{m(\rho-nc-1)}$ for $ nc+1 \leq \rho \leq nb.$
Further, if $nb+1 \leq \rho \leq np^s,$ then the RT weight of the codeword $Q(x) \in \mathcal{C}$ is $\rho$ if and only if one of the following two conditions are satisfied: (i)  $\text{ deg }A_Q(x)=\rho-nb-1$ and $B_Q(x)$ is either 0 or $\text{deg }B_Q(x) \leq \rho-nc-1$ and  (ii) $A_Q(x)$ is either 0 or $\text{deg }A_Q(x) \leq \rho-n b-2$ and $\text{deg }B_Q(x) =\rho-nc-1.$ From this, we get $\mathcal{A}_\rho=(p^{2m}-1)p^{m(2\rho-nb-nc-2)}$ for $nb+1 \leq \rho \leq np^s.$

To prove (e),  let $\mathcal{C}=\langle(x^{n}-\alpha_0)^a+u (x^{n}-\alpha_0)^{t_1} D_1(x)+u^2 (x^{n}-\alpha_0)^{t_2} D_2(x),u (x^{n}-\alpha_0)^{b}+u^2(x^{n}-\alpha_0)^{\theta}V(x),\\ u^2 (x^{n}-\alpha_0)^{c}\rangle.$  Here we see that each codeword $Q(x)\in \mathcal{C}$ can be uniquely expressed as $Q(x)= ((x^{n}-\alpha_0)^a+u (x^{n}-\alpha_0)^{t_1} D_1(x)+u^2 (x^{n}-\alpha_0)^{t_2} D_2(x))M_Q(x)+(u(x^{n}-\alpha_0)^{b}+u^2(x^{n}-\alpha_0)^{\theta}V(x))N_Q(x)+u^2 (x^{n}-\alpha_0)^{c}W_Q(x),$ where $M_Q(x),N_Q(x),W_Q(x) \in \mathbb{F}_{p^m}[x]$ satisfy  $ \text{ deg }M_Q(x) \leq n(p^s-a)-1$ if $M_Q(x)\neq 0,$ $ \text{ deg }N_Q(x) \leq n(p^s-b)-1$ if $N_Q(x)\neq 0,$  and $ \text{ deg }W_Q(x) \leq n(p^s-c)-1$ if $W_Q(x)\neq 0.$   From this, we see that if $nc+1 \leq \rho \leq nb,$ then the codeword $Q(x) \in \mathcal{C}$ has  RT weight $\rho$ if and only if $M_Q(x)=N_Q(x)=0$ and $\text{ deg }W_Q(x)=\rho-nc-1.$ This implies that $\mathcal{A}_\rho=(p^{m}-1)p^{m(\rho-nc-1)}$ for $ nc+1 \leq \rho \leq nb.$
Further, if $nb+1 \leq \rho \leq na,$ then the RT weight of the codeword $Q(x) \in \mathcal{C}$ is $\rho$ if and only if $M_Q(x)=0$  and one of the following two conditions are satisfied: (i) $\text{ deg }N_Q(x)=\rho-nb-1$ and $W_Q(x)$ is either 0 or $\text{deg }W_Q(x) \leq \rho-1-nc;$   and (ii) $N_Q(x)$ is either 0 or $\text{deg }N_Q(x) \leq \rho-nb-2$ and $\text{ deg }W_Q(x) =\rho-nc-1.$ This implies that $\mathcal{A}_\rho=(p^{2m}-1)p^{m(2\rho-n\omega-n\mu-2)}$ for $ nb+1 \leq \rho \leq na.$ 
Next let  $na+1 \leq \rho \leq np^s.$ Here the RT weight of the codeword $Q(x) \in \mathcal{C}$ is $\rho$ if and only if one of the following three conditions are satisfied: (i) $\text{ deg }M_Q(x)=\rho-na-1,$ $N_Q(x)$ is either 0 or $\text{deg }N_Q(x) \leq \rho-nb-1$  and $W_Q(x)$ is either 0 or $\text{deg }W_Q(x) \leq \rho-nc-1;$ (ii) $M_Q(x)$ is either 0 or $\text{deg }M_Q(x) \leq \rho-na-2,$ $\text{ deg }N_Q(x) =\rho-nb-1$ and  $W_Q(x)$ is either 0 or $\text{deg }W_Q(x) \leq \rho-nc-1;$ and (iii) $M_Q(x)$ is either 0 or $\text{deg }M_Q(x) \leq \rho-na-2,$ $N_Q(x)$ is either 0 or $\text{deg }N_Q(x) \leq \rho-nb-2$ and $\text{ deg }W_Q(x) =\rho-nc-1.$  This implies that $\mathcal{A}_\rho=(p^{3m}-1)p^{m(3\rho-na-nb-nc-3)}$ for $ na+1 \leq \rho \leq np^s.$
 
This completes the proof of the theorem.\vspace{-2mm}\end{proof}
\vspace{-4mm}\section{Conclusion and Future work}\label{con}\vspace{-2mm}
Let $p$ be a prime,  $n,s,m$ be positive integers  with $\gcd(n,p)=1,$  $\mathbb{F}_{p^m}$ be the finite field of order $p^m,$  and let $\mathcal{R}=\mathbb{F}_{p^m}[u]/\langle u^3 \rangle $ be the finite commutative chain ring with unity.  Let  $\alpha,\beta, \gamma \in \mathbb{F}_{p^m}$ and $\alpha \neq 0.$ When $\alpha$ is an $n$th power of an element in $\mathbb{F}_{p^m}$ and $\beta \neq 0, $ one can determine all $(\alpha+\beta u +\gamma u^2)$-constacyclic codes of length $np^s$ over $\mathcal{R}$ by applying the results derived in Cao \cite{cao1} and by establishing a ring isomorphism from $\mathcal{R}[x]/\langle x^{np^s}-1-\alpha^{-1}\beta u -\alpha^{-1}\gamma u^2 \rangle $ onto $\mathcal{R}[x]/\langle x^{np^s}-\alpha-\beta u -\gamma u^2\rangle.$ However, when $\alpha$ is not an $n$th power of an element in $\mathbb{F}_{p^m}$ and $\beta \neq 0,$ algebraic structures of all  $(\alpha+\beta u +\gamma u^2)$-constacyclic codes of length $np^s$ over $\mathcal{R}$ and their dual codes were not established. In this paper, we determined all $(\alpha+\beta u +\gamma u^2)$-constacyclic codes of length $np^s$ over $\mathcal{R}$ and their dual codes when $\beta \neq 0.$ We also considered the case $\beta=0$ in this paper, and we determined  all $(\alpha+\gamma u^2)$-constacyclic codes of length $np^s$ over $\mathcal{R}$ and their dual codes. We also listed some isodual $(\alpha+\beta u+\gamma u^2)$-constacyclic codes of length $np^s$ over $\mathcal{R}$ when the binomial $x^n-\alpha_0$ is irreducible over $\mathbb{F}_{p^m}.$ 

In another work \cite{sharma3}, we obtained Hamming distances, RT distances and RT weight distributions of $(\alpha+\beta u+\gamma u^2)$-constacyclic codes of length $np^s$ over $\mathcal{R}$ when the binomial $x^n-\alpha_0$ is irreducible over $\mathbb{F}_{p^m}$ and $\beta$ is non-zero. In this paper, we considered the case $\beta=0$ and we explicitly determined these parameters for all $(\alpha+\gamma u^2)$-constacyclic codes of length $np^s$ over $\mathcal{R},$ provided the binomial $x^n-\alpha_0$ is irreducible over $\mathbb{F}_{p^m}.$

 This work completes the problem of determination of all repeated-root constacyclic codes of arbitrary lengths over $\mathcal{R}$ and their dual codes. It would be interesting  to determine their Hamming distances, RT distances and RT weight distributions in the case when the binomial $x^n-\alpha_0$ is reducible over $\mathbb{F}_{p^m}.$ Another interesting problem would be to study their duality properties and to determine their homogeneous distances.

\end{document}